\documentclass{amsart}

\newtheorem{teo}{Theorem}[section] 
\newtheorem{kowski}[teo]{Algorithm}
 
\newtheorem{lemma}[teo]{Lemma} 
 
\newtheorem{prop}[teo]{Proposition}
\newtheorem{corollario}[teo]{Corollary}
\theoremstyle{definition} 
\newtheorem{definiz}[teo]{Definition} 
\newtheorem{example}[teo]{Example}
\newtheorem{oss}[teo]{Remark}
\newenvironment{proof1}{\noindent\textit{Proof}}{\hfill$\square$\medskip\\}

\usepackage{epsfig,amsfonts,amssymb}
\usepackage{color}
\usepackage[all]{xy}

\theoremstyle{remark} 

\numberwithin{equation}{section}

\begin{document}
\title[Quaternionic Tori]{On Quaternionic Tori and their Moduli Space}
 \author[C. Bisi, G. Gentili]
{Cinzia Bisi $^\ast$, Graziano Gentili $^\ast$}
\date{\today}
\thanks{\rm $^\ast$ Both authors are supported by Progetto MIUR di
Rilevante Interesse Nazionale {\it ``Variet\`a reali e complesse: geometria, topologia e analisi armonica" } and  by GNSAGA of INdAM. The first author is also supported by Progetto FIRB {\it Geometria differenziale e teoria geometrica delle funzioni}}
\address{First author: Dipartimento di Matematica ed Informatica, Universit\'a di Ferrara,
Via Machiavelli n.35, Ferrara, 44121
Italy}
\vspace{1cm}

\email{bsicnz@unife.it}
\email{cinzia.bisi@unife.it}

\address{Second author: Dipartimento di Matematica e Informatica ``U. Dini", Universit\'a di
Firenze, Viale Morgagni 67/A, 50134 Firenze, Italy}
\vspace{1cm}

\email{gentili@math.unifi.it}


\begin{abstract}
Quaternionic tori are defined as quotients of the skew field $\mathbb{H}$ of quaternions  by rank-4 lattices. Using  slice regular functions, these tori are endowed with natural structures of quaternionic manifolds (in fact quaternionic curves), and a fundamental region in a $12$-dimensional real subspace is  then constructed to classify them up to biregular diffeomorphisms. 
The points of the moduli space correspond to  suitable \emph{special} bases of rank-4 lattices, which are studied with respect to the action of the group $GL(4, \mathbb{Z})$, and up to biregular diffeomeorphisms. All tori with a non trivial group of biregular automorphisms - and all possible groups of their biregular automorphisms - are then identified, and recognized to correspond to five different subsets  of boundary points of the moduli space.
\end{abstract}

\maketitle
\section{Introduction}
A new notion of regularity for quaternion-valued functions of a quaternionic variable was introduced in 2006, in \cite{GS06,GS}. The newly defined class of (slice) regular functions has already proved to be interesting as a quaternionic counterpart of complex holomorphic functions.
In this quaternionic setting, a Casorati-Weierstrass Theorem was proved in \cite{stoppato} and it  allowed the study of the group $Aut(\mathbb{H})$ of all biregular transformations of the space of quaternions $\mathbb{H}$. This group turned out to coincide with  the group of all affine transformations of $\mathbb{H}$ of the form $q\mapsto qa+b$, with $a, b\in \mathbb{H}$ and $a\neq 0$. As we can see, notwithstanding the fact that the algebraic, abstract structure of the group $Aut(\mathbb{C}^2)$ of  biholomorphic transformations of $\mathbb{C}^2$ is still unknown, that of biregular transformations of $\mathbb{H}\cong \mathbb{C}^2$ inherits the simplicity of the group $Aut(\mathbb{C})$.

The fact that all quaternionic regular affine transformations of $\mathbb{H}$ form a group under composition, the group $Aut(\mathbb{H})$, permits the direct construction of a class of natural \emph{quaternionic manifolds} (actually \emph{quaternionic curves}): the quaternionic tori. These tori are studied in the present paper. Together with the quaternionic projective spaces, \cite{salamon}, and the Hopf surfaces, \cite{AB18}, they are among the few directly constructed quaternionic manifolds, and bear with them the genuine interest that accompanies any analog of elliptic complex Riemann surfaces.

In this paper we construct quaternionic tori, realized as quotients of $\mathbb{H}$ with respect to rank-$4$ lattices, and endow them  with  natural structures of quaternionic $1$-dimensional manifolds. We then use the basic features of quaternionic regular maps to characterize biregularly diffeomorphic tori  by means of properties of their generating lattices; this approach introduces into the scenery the group $GL(4, \mathbb{Z})$, that  plays a fundamental role in this context. In fact the use of classical results on the reduction of Gram matrices, based on the Minkowski-Siegel Reduction Algorithm, allows us to express  any ``normalized'' rank-$4$ lattice of $\mathbb{H}$ in terms of a generating \emph{special basis}. These special bases parameterize the classes of equivalence of biregular diffeomorphism of quaternionic tori, and are used to define a \emph{fundamental set} (see \eqref{esse}) for this equivalence relation, as the subset of $\mathbb{H}^3\cong \mathbb{R}^{12}$
\[
\mathcal{M}=\{ (v_2,v_3,v_4) \in \mathbb{H}^3: \{1, v_2,v_3,v_4 \} \hbox{ is a special basis} \}.
\]
We will not define a special basis here in the Introduction (see Definition \ref{special basis}), but we want to say that special bases have properties that urge a comparison with the complex case of elliptic curves, like the following:
if $\{1, v_2,v_3,v_4 \}$ is a special basis of a rank-$4$ lattice, i.e., if $( v_2,v_3,v_4)\in \mathcal{M}$, then, in particular
\begin{enumerate}
\item $1\leq \langle v_2, v_2\rangle \leq  \langle v_3, v_3\rangle \leq  \langle v_4, v_4\rangle$;
\item $-\frac{1}{2} \leq Re (v_k) \leq \frac{1}{2}$, for all $k=2,3,4$;
\item $-\frac{1}{2} \langle v_l, v_l\rangle \leq  \langle v_k, v_l\rangle \leq \frac{1}{2} \langle v_l, v_l\rangle$, for all $(k,l)\in \{2,3,4\} \times \{2,3,4\}$ such that $l\neq k$.
\end{enumerate}
The fundamental set $\mathcal M\subset\mathbb H^3$ of the equivalence relation of biregular diffeomorphism among quaternionic tori has some boundary points that are equivalent. In fact there are different elements belonging to the boundary $\partial \mathcal{M}$ of the fundamental set,  that correspond to the same torus; as an example we can take the distinct points $(i, j, k)$ and $(j, i, k)$: the two special bases $\{1, i, j, k\}$ and $\{1, j, i, k\}$ generate the same lattice (the ring of Lipschitz quaternions) and hence the same torus.
 However, in \eqref{tau} we define  the  proper subset $\mathcal{T}$ of the fundamental set  $\mathcal M$, which coincides with the interior of $\mathcal M$,  and which is a \emph{moduli space} for the subset of equivalence classes of the so called \emph{tame} tori.  The complete quotient of the boundary $\partial \mathcal{M}$, with respect to the equivalence relation of biregular diffeomorphism  of the corresponding tori, is still unknown. However, as it happens in the complex case of elliptic curves, the classification of all the boundary tori of $\partial \mathcal{M}$ having non trivial groups of biregular automorphisms 
is an important step towards the understanding of the subtle features of the geometry of the moduli space.

In this perspective, by exploiting the classification of the finite subgroups of unitary quaternions,  we identify  all the groups that can play the role of groups of biregular automorphisms of tori, i.e.,
\begin{equation*} 
2\mathbb{T}, \ \ \ 2D_{4}, \ \ \ 2D_{6}, \ \ \ 2C_1, \ \ \ 2C_2, \ \ \ 2C_3,
\end{equation*}
called respectively \emph{tetrahedral, $8$-dihedral, $12$-dihedral, trivial-cyclic, cyclic-dihedral} and \emph{cyclic} group. We then find those points of the boundary of the fundamental set $\mathcal M$ which correspond to tori whose group of biregular automorphisms either contains, or is isomorphic to, one of the groups listed above. 
We then prove
\begin{teo}\label{armonici} The only quaternionic tori $T$ with a non trivial group of biregular automorphisms $Aut(T)\neq \{\pm1\}$ correspond to the following elements of $\mathcal M$:
\begin{itemize}
\item $(I, \alpha_3, \alpha_3 I)\in \mathcal M$ \ \textrm{with} \  $|\alpha_3|\geq 1 $  and $I^2=-1$\ \textnormal{[}\ $2C_2 \subseteq Aut(T)$\ \textnormal{]};
\item $(e^{\frac{\pi I}{3},} \alpha_3, \alpha_3 e^{\frac{\pi I}{3}})\in \mathcal M$  \ \textrm{with} \  $|\alpha_3|\geq 1 $  and $I^2=-1$  \ \textnormal{[}\ $2C_3 \subseteq Aut(T)$\ \textnormal{]}.
\end{itemize}
\end{teo}
\noindent The most structured groups of biregular automorphisms appear in  the tori with the richest symmetries: examples are the torus generated by  the lattice of \emph{Lipschitz quaternions} and the one generated by the lattice of \emph{Hurwitz quaternions},  whose groups of automorphisms are  the $8$-dihedral group $2D_4$ and the tetrahedral group $2\mathbb T$, respectively.
Notice that the complex counterpart of the tori listed in Theorem \ref{armonici} consists of the \emph{ harmonic} and \emph{equianharmonic} tori, having moduli $i$ and $e^{\frac{\pi i}{3}}$, respectively.

Appendices present a computational approach to the study of  the modulus of a torus:  for example, in the last appendix,  an algorithm is produced that checks if a given basis of a lattice is tame. \\
\indent To conclude the Introduction, we point out that the theory of slice regular functions, presented in detail in the monograph \cite{libro}, has been applied to the study of a non-commutative functional calculus, (see for example the monograph \cite{librofc} and references therein) and to address the problem of the construction and classification of orthogonal complex structures in open subsets of the space $\mathbb{H}$ of quaternions (see \cite{GenSalSto}).
Recent results of geometric theory of regular functions appear in \cite{BG}, \cite{BG2}, \cite{BS}, \cite{BS2}, \cite{BS17},
\cite{BohrDGS}, \cite{BlochDGS}, 
 \cite{GPSlice}, \cite{GPseveral}, \cite{GPS}.
Paper \cite{VV} is strictly related to the topic of the present article.

\section{Preliminary results}
In this section we will briefly present those results on slice regular functions that are essential for what follows.

The $4$-dimensional real algebra of quaternions is denoted by $\mathbb{H}$. 
An element $q$ in $\mathbb{H}$ can be expressed in terms of the standard basis, denoted by $\{1,i,j,k\}$, as $q=x_0+x_1i+x_2j+x_3k$, 
where $i,j,k$ are imaginary units, $i^2=j^2=k^2=-1$, related by the multiplication rule $ij=k$.  To every non-real quaternion $q\in \mathbb{H} \setminus \mathbb{R}$ we can associate an imaginary unit, with the map
\[q \mapsto I_q=\frac{ Im(q)}{ | Im(q)|}.\]
If instead $q\in\mathbb{R}$, we can set $I_q$ to be any arbitrary imaginary unit.   
In this way, for any $q\in\mathbb{H}$ there exist, and are unique, $x,y\in\mathbb{R}$, 
with $y\ge 0$ ($y=0$ if $q\in\mathbb{R}$), such that
\[q=x+yI_q.\]  
The set of all imaginary units is denoted by $\mathbb{S}$, 
\[\mathbb{S}=\{q \in \mathbb{H} \, | \, q^2=-1\}\]
and, from a topological point of view, it is a $2$-dimensional sphere sitting in the $3$-dimensional space of purely imaginary quaternions. The symbol $\mathbb{B}$ will denote the open unit ball $\{q\in \mathbb{H}: |q|<1 \}$ of the space $\mathbb{H}$ of quaternions, and the $3$-sphere of all the points of its  boundary $\partial \mathbb{B}$ will be denoted by $\mathbb{S}^3$.

\begin{oss}\label{omotetia} It is worthwhile recalling that the two possible multiplications of a quaternion $q$ by any fixed element $u\in \partial \mathbb B$, that is $q \mapsto qu$ and $q \mapsto uq$,  both represent Euclidean rigid motions (rotations) of $\mathbb H \cong \mathbb R^4$.
Moreover, if $v\in \mathbb H^*$ is any fixed non zero quaternion, then both multiplications, $q \mapsto qv$ and $q \mapsto vq$,  can be decomposed as the composition of a Euclidean rigid motion of $\mathbb H \cong \mathbb R^4$ and a multiplication by a strictly positive real number: the reason is immediate, since  $v=|v|\frac{v}{|v|}$ with $\frac{v}{|v|}\in \partial \mathbb B$.
\end{oss}

To each element $I$ of $\mathbb{S}$ there corresponds a copy of the complex plane, namely $L_I=\mathbb{R}+ I\mathbb{R}\cong \mathbb{C}$. 
All these complex planes, also called {\em slices}, intersect along the real axis, and their union gives back the space of quaternions,
\[\mathbb{H}=\bigcup_{I\in \mathbb{S}} \left(\mathbb{R}+I \mathbb{R}\right)=\bigcup_{I\in\mathbb{S}}L_I.\] 
Since $L_I=\mathbb{R}+ I\mathbb{R}\cong \mathbb{C}$, with the symbol $e^{\alpha I},$ we will mean $\cos \alpha + I \sin \alpha.$
The following definition appears in \cite{GS06,GS}.
\begin{definiz}
Let $\Omega$ be a domain in $\mathbb{H}$ and let $f : \Omega \to \mathbb{H}$ be a function. For all $I \in \mathbb{S}$ 
let us consider 
$\Omega_I = \Omega \cap L_I$ and $f_I = f_{|_{\Omega_I}}$. 
The function $f$ is called \emph{(slice) regular} if, for all $I \in \mathbb{S}$, the restriction $f_I$ has continuous derivatives and the function $\bar \partial_I f : \Omega_I \to \mathbb{H}$ defined by
$$
\bar \partial_I f (x+Iy) = \frac{1}{2} \left( \frac{\partial}{\partial x}+I\frac{\partial}{\partial y} \right) f_I (x+Iy)
$$
vanishes identically.
\end{definiz}
\noindent The same articles introduce the \emph{Cullen (or slice) derivative} $\partial_cf$ of a slice regular function $f$ as
\begin{equation}\label{cullen}
\partial_cf(x+Iy)=\frac{1}{2}\left(\frac{\partial}{\partial x}-I\frac{\partial}{\partial y}\right)f(x+Iy)
\end{equation}
for $I \in \mathbb{S},\ x,y \in \mathbb{R}.$ 
\begin{oss}\label{derivReg}
The Cullen derivative of a slice regular function turns out to be still a slice regular function (see, e.g, \cite[Definition 1.7, page 2]{libro}, \cite{GS}).
\end{oss}
Using the Cullen derivative, it is possible to characterize the slice regular functions defined in the entire space $\mathbb{H}$, or on a ball $B(0, R)=\{q\in \mathbb{H} : |q|<R\}$ centered at $0\in \mathbb{H}$, as follows (see, e.g., \cite{libro}).

\begin{teo}\label{expansion}
A function $f$ is regular in $B(0,R)$ if and only if $f$ has a power series expansion
\[f(q)=\sum_{n \geq 0}q^na_n \quad\text{with} \quad a_n=\frac{1}{n!}\frac{\partial^n f}{\partial x^n}(0)\]
converging in $B(0, R)$. Moreover its Cullen derivative can be expressed as
\[ \partial_c f(q)=\sum_{n \geq 0}q^nb_n \quad\text{with} \quad b_n=\frac{1}{n!}\frac{\partial^{n+1} f}{\partial x^{n+1}}(0)\]
in $B(0,R)$. 
\end{teo}  

The existence of the power series expansion yields a Liouville Theorem, that we will use in the sequel:

\begin{teo}[Liouville]\label{GS}
Let  $f: \mathbb{H} \to \mathbb{H}$ be slice regular. If $f$ is bounded  then $f$ is constant.
\end{teo}

\section{Lattices in the space  of quaternions}

Let $\omega_{1},\ldots, \omega_{m}$ (with $m\leq 4$) be $\mathbb{R}$-linearly
independent vectors in $\mathbb{H}$. 

\begin{definiz}
 The additive
subgroup of $(\mathbb{H}, +)$ generated by $\omega_{1},\ldots, \omega_{m}$
is called a {\sl rank-$m$ lattice, generated} by $\omega_{1},\ldots, \omega_{m}$.
\end{definiz}

We will focus our attention on (topologically) discrete subgroups of $(\mathbb{H},
+)$, for which the following result holds:

\begin{lemma} \label{LemmaCauchy}
Let $M$ be a discrete (infinite) subgroup of $(\mathbb{H}, +)$. Then $M$ has no accumulation points.
\end{lemma}
\begin{proof1} Since $M$ is discrete, there are no accumulation points of $M$ belonging to
$M$.
By contradiction, assume that there exists an accumulation point $q$ of
$M$, belonging to $\mathbb{H}\setminus M$. Then there exists  a sequence
$\{ q_n \}_{n \in \mathbb{N}} \subseteq M$ 
converging to $q$. 
Since $\{ q_n \}_{n \in \mathbb{N}}\subseteq M$ is a Cauchy sequence, for all $m \in \mathbb{N},$
there exist $r_m,\,\,\, s_m \in \mathbb{N}$ such that $q_{r_m} \neq q_{s_m}$ and that
\begin{equation} \label{eq1}
0<|q_{r_m} -q_{s_m}| < \frac{1}{m}.
\end{equation} 
If we define $\alpha_m=q_{r_m} - q_{s_m} \in M,$ inequality
(\ref{eq1}) would imply that $0 \in M$ is not isolated, 
being the limit of $\{ \alpha_m \}_{m \in \mathbb{N}}$ when $m \to +\infty.$
\end{proof1}

As a straightforward consequence we obtain:

\begin{corollario} Let $M$ be a subgroup of $(\mathbb{H}, +)$, let $R\in \mathbb{R}^+$ 
and let $\overline{B(0,R)}$ be the
closure of $B(0,R)$. Then $M$ is
discrete if and only if, for all $R\in \mathbb{R}^+$, the intersection $\overline{B(0,R)}\cap M$ 
is a finite set.
\end{corollario}

The following classical characterization of discrete subgroups of $(\mathbb{H}, +)$
will be used as a basic fact in the sequel (for a proof see, e.g.,
\cite[Theorem 6.1, page 136]{St}). 

\begin{teo}\label{discreteness}
A subgroup of $(\mathbb{H}, +)$ is a lattice if and only if it is
discrete. 
\end{teo}

This last theorem implies that the study of all possible quotient
spaces of $(\mathbb{H}, +)$ with respect to a discrete additive
subgroup $M$ is reduced to the case in which $M$ is a lattice.  
With the aim of classifying these quotients, let $T^{m}$ denote the direct product of $m$ copies of the unit circle
$S$ of $\mathbb{R}^{2}$, and call it the {\sl $m$-dimensional torus}.

It is well known that, given the rank $m$ 
of a lattice in $\mathbb{H}\cong \mathbb{R}^4$, there exists only ``one'' quotient, up to real
diffeomorphisms (see, e.g., \cite[Theorem 6.4, page 140]{St}): 

\begin{teo} Let $L$ be a rank-$m$ lattice in $\mathbb{H}$ (with $m\leq 4$).
    Then the group $\mathbb{H}/ L$ is isomorphic to $T^{m}\times
    \mathbb{R}^{4-m}$.
\end{teo}

The case of a rank-$4$ lattice is the one in which the
quotient originates ``the'' real
$4$-dimensional torus:

\begin{corollario} Let $L$ be a rank-$4$ lattice in $\mathbb{H}$.
    Then the group $\mathbb{H}/ L$ is isomorphic to $T^4$.
\end{corollario}

As we can see, up to real diffeomorphisms the classification is quite simple.
Following the guidelines of the classical theory of complex elliptic functions
we will work at the classification of 
$4$-(real)-dimensional, quaternionic tori, up to biregular
diffeomorphisms. In the next section we will define slice quaternionic
structures on tori, that will be the object of our classification.

\section{A regular quaternionic structure on a $4$-(real)-dimensional torus}


Since quaternionic regular affine transformations form a group with respect to composition, and analogously to what happens for complex tori, the field $\mathbb H$ induces on a quaternionic torus $T^4=\mathbb H/L$ a structure of quaternionic manifold. This
structure will be called \emph{a regular quaternionic structure}
or simply a \emph{quaternionic} structure on $T^4$. From now on, a torus $T^4$ endowed with a quaternionic structure will be called a \emph{quaternionic torus}. Moreover, the $4$-(real)-dimensional torus $T^4$ will always be denoted simply by $T$. 

To construct such a structure, we
will first of all consider the classical atlas $\mathcal{U}$ of the real torus
$T$,  and then adopt the procedure used in the complex case (see,
e.g., \cite{bredon, FK, Ves}).
Let $L$ be a rank-$4$ lattice  of $\mathbb{H}$, generated by
$\omega_{1}, \omega_{2}, \omega_{3}, \omega_{4}$. Consider the canonical projection $\pi: \mathbb{R}^{4}\cong \mathbb{H} \to
\mathbb{H}/ L=T$ and, for any $p\in \mathbb{H}$,  an open
neighborhood $U_{p}$ of $p$ small enough to make $\pi_{{|_{U_{p}}}}$ an
homeomorphism of $U_{p}$ onto its image $\pi(U_{p})$. The atlas $\mathcal U$ will consist of the local coordinate systems
$\left\{\left( \pi(U_{p}), (\pi_{{|_{U_{p}}}})^{-1}\right)\right\}_{p\in \mathbb{H}}$. If 
we suppose that, for $p, q\in \mathbb{H}$, the intersection
$\pi(U_{p})\cap \pi(U_{q})$ is (open and) connected, then the change of 
coordinates is such that $\pi_{{|_{U_{p}}}}^{-1}\circ
\pi_{{|_{U_{q}}}}(x) = x +\sum_{l=1}^{4}n_{l}\omega_{l}$ for fixed
$n_{1}, n_{2}, n_{3}, n_{4}$. Hence the change of coordinates is a 
 regular function. Therefore we obtain a quaternionic structure
on $T$.
Using the classical approach,  \emph{regular maps} between quaternionic tori can be defined in the natural manner, as well as \emph{biregular diffeomorphisms} between quaternionic tori, and \emph{biregular automorphisms} of a quaternionic torus. We can then proceed to study the quaternionic tori up to biregular diffeomorphisms, and give the following:

    \begin{definiz}\label{equivalent-tori}
     If there is a biregular diffeomorphism of a $4$-(real)-dimensional torus $T_1$ onto a ($4$-(real)-dimensional) torus $T_2,$ we will say that the two tori are \emph{equivalent}.
    \end{definiz}
To proceed,  we recall that the group $Aut(\mathbb{H})$ of biregular transformations (or
automorphisms) 
of $\mathbb{H}$
consists of all slice regular affine transformations, 
that is
$$
Aut(\mathbb{H})=\{f(q)=qa+b : a,b \in
\mathbb{H}, a\neq 0\}
$$
(see 
\cite{stoppato}).

The result stated in the next proposition has a complete analog in the complex setting, \cite[Theorem 4.1, page 10]{FK}. 
Nevertheless we 
will produce a proof, to acquire familiarity with the quaternionic
environment and to establish notations to be used.

\begin{prop} \label{lifting}
Let $L_{1}$ and $L_{2}$ be  two rank-$4$ lattices in
$\mathbb{H}$, let $\pi_{1}: \mathbb{H} \to \mathbb{H}/L_{1}= T_{1}$
and $\pi_{2}: \mathbb{H} \to \mathbb{H}/L_{2}= T_{2}$ be the
projections on the quotient tori. For any $F \in Aut(\mathbb{H})$ such
that $F(L_{1}) = L_{2}$
there exists a biregular diffeomorphism $f$ of $T_{1}$ onto  $T_{2}$
which allows the equality $f\circ \pi_{1} = \pi_{2} \circ F$. Conversely, for any biregular 
diffeomorphism $f$ of $T_{1}$ onto  $T_{2}$, there exists 
$F \in Aut(\mathbb{H})$ such that $f\circ \pi_{1} = \pi_{2} \circ F$
and $F(L_{1})= L_{2}$.
\end{prop}
\begin{proof}

%
Let $F(v)=va+b$. Since $ 0\in L_{1}$ we have $F(0)=b
\in L_{2}$ and hence we can suppose $b=0$. 
By definition of regular map between tori, to show that $F(v)=va$
induces a biregular diffeomorphism $f$ of $T_{1}$ onto 
$T_{2}$,
\begin{equation}\label{diagramma}
\xymatrix{\mathbb{H} \ar[r]^{F} \ar[d]_{\pi_1} & \mathbb{H}
\ar[d]^{\pi_2} \\
T_{1} \ar[r]_{f} & T_{2}}
\end{equation}
it is enough to show that $q\sim p$ implies $F(q)\sim
F(p)$. Indeed, if $q\sim p$ then $q-p \in L_{1}$ and hence $F(q)-F(p) = 
qa-pa =(q-p)a = F(q-p) \in L_{2}$.

To prove the converse statement, we start by recalling that the map $f: T_{1}
\to T_{2}$ lifts to a continuous map $F: \mathbb{H} \to \mathbb{H}$, in such a way
that the diagram (\ref{diagramma}) commutes. 
Moreover the map $F$ is regular
since $f$ is a regular map of $T_1$ onto $T_2$.

For any $\lambda
\in L_{1}$, consider $G_{\lambda}(q)= F(q+\lambda) - F(q)$.
Since $F$ lifts a map between the quotients, $F$
maps $L_1$-equivalent points into $L_2$-equivalent points. 
Hence 
the image of $G_{\lambda}$ is contained in the (discrete, see Theorem \ref{discreteness}) lattice $L_{2}$
and, being continuous, is therefore constant. At this point it is clear that, taking the
Cullen derivative, we obtain $\partial_{c}F(q+\lambda)=\partial_{c}F(q)$,
for all $q\in \mathbb{H}$. Thus the map $\partial_{c}F$ is regular
(see Remark \ref{derivReg}) and $L_{1}$-periodic, which makes it bounded. By 
the Liouville Theorem for regular functions (see Theorem \ref{GS}) the Cullen
derivative $\partial_{c}F$ of $F$ is constant. Since $F$ expands as
a power series (see Theorem \ref{expansion})
$$
F(q)=\sum_{n\in \mathbb{N}}q^{n}\frac{1}{n!}\frac{\partial^{n} F}{\partial
x^{n}}(0)
$$ 
converging in the entire $\mathbb{H}$, we obtain (again by Theorem \ref{expansion})
$$
\partial_{c}F(q)=\sum_{n\in
\mathbb{N}}q^{n}\frac{1}{n!}\frac{\partial^{n+1} F}{\partial
x^{n+1}}(0) = \frac{\partial F}{\partial
x} (0)
$$
and hence
$$
F(q)=F(0) + q\frac{\partial F}{\partial
x} (0) = b+qa
$$
is a first degree regular polynomial. Again, since $F$ lifts a map
between quotients, necessarily $L_{1}a \subseteq L_{2}$. If the
inclusion $L_{1}a \subset L_{2}$ is proper, then $f$ is not
injective: indeed if some $q\in L_{2}$ satisfies $qa^{-1}\not \in
L_{1}$ then there exists $p_1, p\in L_1$ and $p_2\in L_2$ such that $(qa^{-1}+p_1)\notin L_{1}$ and $f(qa^{-1}+p_1) = \pi_{2}(q + p_1a +b) = \pi_{2}( p_2) 
= f(p)$, with $qa^{-1}+p_1\neq p$.

Now we know that $L_{1}a = L_{2}$, that is $L_{2}a^{-1}=L_{1}$. The
map $F^{-1}: \mathbb{H} \to \mathbb{H}$ defined by $F^{-1}(w) =
(w-b)a^{-1}$ induces the map $f^{-1}: T_{2} \to T_{1}$. Indeed
$$
f^{-1}(f(q+L_{1})) = f^{-1}(qa+b+L_{2}) = (qa+ L_{2})a^{-1} = q+ L_{2}a^{-1}
= q+L_{1}.
$$
This concludes the proof.
\end{proof}

\section{Equivalence of quaternionic tori}

To classify
the $4$-(real)-dimensional, quaternionic tori, up to biregular
diffeomorphisms, 
we start with the following:

\begin{teo} \label{teoequiv}
Two rank-$4$ lattices $L_{1}, \ L_{2}$ of  the space $\mathbb{H},$ generated
respectively by the bases $\{\alpha_1,\alpha_{2}, \alpha_{3}, \alpha_{4}\}$ and $\{\omega_1, \omega_{2}, 
\omega_{3}, \omega_{4}\}$, determine equivalent tori  $T_1,  T_2$ if and only if there
exist  $a\in \mathbb{H}^{*} = \mathbb{H}\setminus \{0\}$ and a linear transformation $A \in GL(4, \mathbb{Z})$ such that
\begin{equation*}
A \left( \begin{array}{l}
        \omega_1 \\
        \omega_{2} \\
	\omega_{3} \\
        \omega_{4} 
	\end{array} \right)
	= 
	\left( \begin{array}{l}
	     \alpha_1 \\
	     \alpha_{2} \\
	     \alpha_{3} \\
	     \alpha_{4} 
	     \end{array} \right)
	a
\end{equation*}
\end{teo}
\begin{proof}
By Proposition \ref{lifting}, if $f$ is a biregular diffeomorphism of
$T_{1}$ onto $T_{2}$,  then there exists a biregular
transformation
$F$ of $\mathbb{H}$ such that the diagram (\ref{diagramma}) commutes.
Since $F$ is biregular on $\mathbb{H},$ then $F(q)=qa+b,$ with $a \in
\mathbb{H}^{*},$ and $b \in \mathbb{H}.$
As we pointed out in the proof of Proposition \ref{lifting}, without
loss of generality, we can
suppose both that $b=0$ and that
the function $F$ maps the set of generators of $L_{1}$ to a set of 
generators of $L_{2}.$ Taking into account that
\begin{equation*}
\left\{ \begin{array}{lll}
         F(\alpha_1) & = & \alpha_1a \\
	 F(\alpha_{2}) & = & \alpha_{2} a \\
         F(\alpha_{3}) & = & \alpha_{3} a \\
	 F(\alpha_{4}) & = & \alpha_{4} a
	 \end{array} \right.
    \end{equation*}    
there exists a matrix 
\begin{equation*}
A=\left( \begin{array}{llll}
          n_{11} & n_{12} & n_{13} & n_{14} \\
	  n_{21} & n_{22} & n_{23} & n_{24} \\
         n_{31} & n_{32} & n_{33} & n_{34} \\ 
	 n_{41} & n_{42} & n_{43} & n_{44}
	 \end{array} \right)
    \end{equation*}     
 with integer entries,   such that     
\begin{equation} \label{sistequiv}
\left\{ \begin{array}{llllllll}
                    \alpha_1a = & n_{11}\omega_1 & + & n_{12} \omega_{2} & + & n_{13} \omega_{3} & + &n_{14} \omega_{4} \\
	 \alpha_{2} a = & n_{21}\omega_1 & + & n_{22} \omega_{2} & + & n_{23} \omega_{3} & +
	 &n_{24} \omega_{4} \\
         \alpha_{3} a = & n_{31}\omega_1 & + & n_{32} \omega_{2} & + & n_{33} \omega_{3} & +
	 &n_{34} \omega_{4} \\ 
	 \alpha_{4} a = & n_{41}\omega_1 & + & n_{42} \omega_{2} & + & n_{43} \omega_{3} & +
	 &n_{44} \omega_{4} 
	 \end{array} \right.
    \end{equation}     
or, more concisely, 
 \begin{equation}\label{alpha=vdoppio}
     	\left( \begin{array}{l}
	     \alpha_1 \\
	     \alpha_{2} \\
	     \alpha_{3} \\
	     \alpha_{4} 
	     \end{array} \right)
	a=
A \left( \begin{array}{l}
        \omega_1 \\
        \omega_{2} \\
	\omega_{3} \\
        \omega_{4} 
	\end{array} \right).
\end{equation}
 The same argument applied in the opposite direction, implies the
 existence of a matrix $B$ with integer entries such that   
    \begin{equation}\label{vdoppio=alpha}
\left( \begin{array}{l}
        \omega_1 \\
        \omega_{2} \\
	\omega_{3} \\
        \omega_{4} 
	\end{array} \right)a^{-1}=
B	\left( \begin{array}{l}
	     \alpha_1 \\
	     \alpha_{2} \\
	     \alpha_{3} \\
	     \alpha_{4} 
	     \end{array} \right)
\end{equation}
    and hence, substituting equation (\ref{vdoppio=alpha}) into
    equation (\ref{alpha=vdoppio}), we get
     \begin{equation*}
     	\left( \begin{array}{l}
	     \alpha_1 \\
	     \alpha_{2} \\
	     \alpha_{3} \\
	     \alpha_{4} 
	     \end{array} \right)
	a=
A B	\left( \begin{array}{l}
	     \alpha_1 \\
	     \alpha_{2} \\
	     \alpha_{3} \\
	     \alpha_{4} 
	     \end{array} \right)a
\end{equation*}
which implies $AB=I_{4}$ and hence that $A$ (and $B$) is such that
$det(A)=\pm 1$, i.e. 
$A$ (and $B$) belongs to $GL(4, \mathbb{Z}).$ 

On the other side, suppose there exists a matrix $A \in GL(4,
\mathbb{Z}),$ of this form:
\begin{equation*}
A=\left( \begin{array}{llll}
          n_{11} & n_{12} & n_{13} & n_{14} \\
	  n_{21} & n_{22} & n_{23} & n_{24} \\
         n_{31} & n_{32} & n_{33} & n_{34} \\ 
	 n_{41} & n_{42} & n_{43} & n_{44}
	 \end{array} \right) 
    \end{equation*}     
such that 
\begin{equation*}
A \left( \begin{array}{l}
        \omega_1\\
        \omega_{2} \\
	\omega_{3} \\
        \omega_{4}
	\end{array} \right)
	= 
	\left( \begin{array}{l}
	\alpha_1\\
	     \alpha_{2} \\
	     \alpha_{3} \\
	     \alpha_{4}
	     \end{array} \right)
	a    
\end{equation*}
then we can compute $F(q)=qa$ in four different {ways}:
\begin{eqnarray*}
    F(q)=q\alpha_1^{-1}(n_{11}\omega_1 + n_{12} \omega_{2} + n_{13} \omega_{3} + n_{14} \omega_{4})\\
    F(q)=q\alpha_2^{-1}(n_{21}\omega_1 + n_{22} \omega_{2} + n_{23} \omega_{3} + n_{24} \omega_{4})\\
    F(q)=q\alpha_3^{-1}(n_{31}\omega_1 + n_{32} \omega_{2} + n_{33} \omega_{3} + n_{34} \omega_{4})\\
    F(q)=q\alpha_4^{-1}(n_{41}\omega_1 + n_{42} \omega_{2} + n_{43} \omega_{3} + n_{44} \omega_{4})
\end{eqnarray*}    
for all $q \in \mathbb{H}.$
Simple computations show that: 
\begin{equation*}
F(q+\alpha_1)=F(q) + n_{11}\omega_1 + n_{12} \omega_{2} + n_{13} \omega_{3} + n_{14} \omega_{4},
\end{equation*}
\begin{equation*}
F(q +\alpha_{2})=F(q) +  n_{21}\omega_1+n_{22} \omega_{2} + n_{23} \omega_{3} + n_{24} \omega_{4},
\end{equation*}
\begin{equation*}
F(q +\alpha_{3})=F(q) +n_{31}\omega_1+ n_{32} \omega_{2} + n_{32} \omega_{3} + n_{34} \omega_{4},
\end{equation*}
\begin{equation*}
F(q +\alpha_{4})=F(q) +n_{41}\omega_1+  n_{42} \omega_{2} + n_{43} \omega_{3} + n_{44} \omega_{4}.
\end{equation*}
Hence $F$ defines a biregular diffeomorphism $f$ between $T_{1}$ and $T_{2}.$
\end{proof}

It is natural at this point to give the following

\begin{definiz} \label{equivlattices}
Two rank-$4$ lattices $L_{1}, \ L_{2}$ of  the space $\mathbb{H}$ are called \emph{equivalent} if the generated quaternionic tori $\mathbb{H}/ L_1$ and $\mathbb{H}/ L_2$ are equivalent. A basis $\{\omega_1, \omega_{2}, 
\omega_{3}, \omega_{4}\}$
of a rank-$4$ lattice $L_{1} $ and a basis $\{\alpha_1,\alpha_{2}, \alpha_{3}, \alpha_{4}\}$  of a rank-$4$ lattice $L_{2}$ are called \emph{equivalent} if $L_1$ and $L_2$ are equivalent lattices, i.e. if (according to Theorem \ref{teoequiv}) there
exist  $a\in \mathbb{H}^{*}$ and a linear transformation $A \in GL(4, \mathbb{Z})$ such that
\begin{equation}\label{formulaequivalenza}
A \left( \begin{array}{l}
        \omega_1 \\
        \omega_{2} \\
	\omega_{3} \\
        \omega_{4} 
	\end{array} \right)
	= 
	\left( \begin{array}{l}
	     \alpha_1 \\
	     \alpha_{2} \\
	     \alpha_{3} \\
	     \alpha_{4} 
	     \end{array} \right)
	a.
\end{equation}
\end{definiz}

Notice that two (different) equivalent bases $\{\omega_1, \omega_{2}, 
\omega_{3}, \omega_{4}\}$   and $\{\alpha_1,\alpha_{2}, \alpha_{3}, \alpha_{4}\}$ of rank-$4$ lattices may generate exactly the same lattice, and hence exactly the same quaternionic torus. This happens when there
exists a linear transformation $A \in GL(4, \mathbb{Z})$ such that
\begin{equation*}
A \left( \begin{array}{l}
        \omega_1 \\
        \omega_{2} \\
	\omega_{3} \\
        \omega_{4} 
	\end{array} \right)
	= 
	\left( \begin{array}{l}
	     \alpha_1 \\
	     \alpha_{2} \\
	     \alpha_{3} \\
	     \alpha_{4} 
	     \end{array} \right)
\end{equation*}
i.e., when $a=1$ in \eqref{formulaequivalenza}.

\section{Minkowski-Siegel Reduction Algorithm: reduced and special bases}\label{sei}

In this section we will specialize to the quaternionic setting the general Minkowski-Siegel Reduction Algorithm presented in \cite[Section 4]{igusa}, \cite[Section 9]{maas}, 
and use it to construct \emph{reduced Gram matrices} and \emph{reduced} bases associated to lattices. In turn, reduced bases will be used to find \emph{special} bases for lattices, useful in the sequel to identify and parameterize equivalence classes of quaternionic tori.

    We explicitly present here some basic facts of this algorithmic construction, both to make the paper as much 
    self-contained as possible, and to have a starting point for the proofs of the
    results that will follow. 
    
    Let $\langle \cdot, \cdot \rangle$ denote the usual scalar product 
    of $\mathbb{R}^{4}$. Let $p=x_0+x_1i+x_2j+x_3k$ and $q=y_0+y_1i+y_2j+y_3k$ be two quaternions. We set, and use in what follows,  
    \begin{equation}\label{scalare}
    \langle p, q \rangle = \sum_{\ell=0}^3 x_\ell y_\ell.
   \end{equation}  
    Let $\{v_1, v_2, v_3, v_4\}$ be a basis of the lattice $L\subset \mathbb{H}\cong \mathbb{R}^4$. For any $u=(n_{1}, n_{2}, n_{3}, n_{4})\in
    \mathbb{Z}^{4}$ the squared norm of the element
    $v=n_{1}v_1+n_{2}v_{2}+n_{3}v_{3}+n_{4}v_{4}\in L$ can be expressed
    by $\langle v, v \rangle= v\ ^{t}v= uS_{0}\ ^{t}u$ where the
    matrix 
    \begin{equation}
	S_{0}= \left( \begin{array}{llll}
          \langle v_{1}, v_{1} \rangle & \langle v_{1}, v_{2} \rangle &\langle v_{1},
	  v_{3} \rangle & \langle v_{1}, v_{4} \rangle \\
	  \langle v_{2}, v_{1} \rangle & \langle v_{2}, v_{2} \rangle & \langle
	  v_{2}, v_{3} \rangle &
	 \langle v_{2}, v_{4} \rangle \\
         \langle v_{3}, v_{1} \rangle &\langle v_{3}, v_{2} \rangle & \langle
	 v_{3}, v_{3} \rangle & \langle v_{3}, v_{4} \rangle \\ 
	 \langle v_{4}, v_{1} \rangle & \langle v_{4}, v_{2} \rangle & \langle
	 v_{4}, v_{3} \rangle &\langle v_{4}, v_{4} \rangle
	 \end{array} \right) 
\end{equation}
    is symmetric and positive definite, and is usually called the
    \emph{Gram matrix} associated to the basis $\{v_{1}, v_{2}, v_{3}, v_{4}\}$.
    In this setting and with the notations established, we will use the following procedure (see, e.g., \cite{maas}, page 122):
    \begin{kowski}[\bf Minkowski-Siegel Reduction Algorithm]\label{kowski}
    \end{kowski}
    \noindent This algorithm acts on a Gram matrix $S_0$ 
and produces a matrix $U=U(S_0)$ belonging to $GL(4, \mathbb Z)$ and a Gram matrix $R=R(S_0)=US_0\ ^tU$. The produced matrix $R$ has, and is in fact characterized by, the properties which follow.

Here are the steps of the algorithm:
    \begin{itemize}
    \item The Gram matrix $S_0$  (of a certain basis $\{v_1, v_2, v_3, v_4\}$) is given.
   \item Consider the function $Q_{1}: \mathbb{Z}^{4} \to \mathbb{R}^+$
    defined as
    \begin{equation*}
Q_{1}(u)= uS_{0}\ ^{t}u.
\end{equation*}
    By our assumption, $Q_{1}$ is (the restriction to $\mathbb Z^4$ of) a positive definite quadratic form, and hence it attains its strictly positive minimum value at a
    point  $u_{1}=(n_{11}, n_{12}, n_{13}, n_{14})\in \mathbb{Z}^{4}$. 
    \item To proceed,  we need 
    to recall that there exist infinitely many matrices of $GL(4,
    \mathbb{Z})$ having the first row equal to $u_{1}$ (see e.g. \cite[Section 4, pages 191-192]{igusa}, 
    \cite[Section 9, pages 122-123]{maas} 
    for a proof of this assertion, and of the analogous ones, used in this algorithm).
    With this in mind, we consider the
    function $Q_{2}$ obtained by restricting $Q_{1}$ to the elements
     $u\in \mathbb{Z}^{4}$ such that there exists a matrix of $GL(4,
    \mathbb{Z})$ having the first two rows equal to $u_{1}$  and $u$,
    respectively. Let  $u_{2}= (n_{21}, n_{22}, n_{23},
    n_{24})\in \mathbb{Z}^{4}$
    be a point in which $Q_{2}$ attains its strictly  positive minimum
    value.
    Up to a
    change of sign, we can assume that $u_{1}S_{0}\ ^{t}u_{2}\geq 0$. 
   \item In the next step, we consider the restriction $Q_{3}$ of
    $Q_{2}$ to the elements
     $u\in \mathbb{Z}^{4}$ such that there exists a matrix of $GL(4,
    \mathbb{Z})$ having the first
    three rows equal to $u_{1}$, $u_{2}$ and
    $u$, respectively. Let  $u_{3}= (n_{31}, n_{32}, n_{33},
    n_{34})\in \mathbb{Z}^{4}$
    be a point in which $Q_{3}$ attains its strictly  positive minimum
    value. Again,  up to a
    change of sign, we can assume that $u_{2}S_{0}\ ^{t}u_{3}\geq 0$. 
   \item  Finally, we take the restriction $Q_{4}$ of $Q_{3}$ to the elements
     $u\in \mathbb{Z}^{4}$ such that there exists a matrix of $GL(4,
    \mathbb{Z})$ having the four rows equal to $u_{1}$, $u_{2}$,
    $u_{3}$ and
    $u$, respectively, and set  $u_{4}= (n_{41}, n_{42}, n_{43},
    n_{44})\in \mathbb{Z}^{4}$
    to be a point in which $Q_{4}$ has a strictly  positive minimum
    value. As before, we can assume $u_{3}S_{0}\ ^{t}u_{4}\geq 0$.
    \item The output of the algorithm consists of the matrix 
     \begin{equation*}
U=U(S_0)=\left( \begin{array}{llll}
    n_{11} & n_{12} & n_{13} & n_{14} \\
	  n_{21} & n_{22} & n_{23} & n_{24} \\
    n_{31} & n_{32} & n_{33} & n_{34} \\ 
	 n_{41} & n_{42} & n_{43} & n_{44}
	 \end{array} \right),
    \end{equation*}
(belonging to $GL(4, \mathbb{Z})$ by construction) and of the Gram matrix $R=US_{0}\ ^{t}U$.
  \end{itemize}  
  \vskip .5cm
As we mentioned, the importance of the Minkowski-Siegel Reduction Algorithm stays in the features of the matrices ($U$ and) $R$ that it produces; in fact, following   \cite[Section 4]{igusa} and \cite[Section 9]{maas}, we can give the next definition. 
 \begin{definiz}\label{reducedGram} If $R$ is a Gram matrix obtained by applying to a given Gram matrix $S_0$ the Minkowski-Siegel Reduction Algorithm  \ref{kowski}, then $R$ is called a \emph{reduced Gram matrix (relative to $S_0$)}. The symbol $\mathcal R$ will denote the set of all reduced Gram matrices.
 \end{definiz}
It turns out that there are two necessary and sufficient conditions that characterize the elements of the set  $\mathcal{R}$ of reduced Gram matrices;  we recall these conditions here (see \cite[equations (1), page 123]{maas}): 
\begin{prop}\label{cnsfundaregion} A Gram matrix $R=(r_{i,j})_{i,j=1, \cdots 4}$ is a reduced Gram matrix if and only if the two following sets of conditions hold:
\begin{itemize} 
\item[B1)] $r_{k,k+1} \ge 0$ for all $k=1, 2, 3;$ \\
\item[B2)] for all 
fixed $k=1,2,3,4$, we have $(n_1,n_2,n_3,n_4) \  R \ \ ^t(n_1,n_2,n_3,n_4) \ge r_{k,k}$  for any integer vector 
$(n_1, n_2, n_3, n_4)$ such that $n_k, \cdots n_4$ are without common divisors.
\end{itemize}
\end{prop}
\noindent We point out, and we will use it in the sequel, the fact that conditions B2) are equivalent to B2)$^\prime$:
\begin{itemize}
\item[B2)$^\prime$] for all fixed $k=1,2,3,4$, if a vector $(n_1,n_2,n_3,n_4)\in \mathbb{Z}^4$ has the property that $(n_1,n_2,n_3,n_4)\  R \  ^t(n_1,n_2,n_3,n_4)<r_{k,k} ,$ then necessarily $n_{k}, \cdots, n_4$ have common divisors.
\end{itemize}

   \begin{oss}\label{chiave} Let $L$ be a rank-$4$ lattice. Consider  any basis $\{v_{1}, v_{2}, v_{3}, v_{4}\}$ of $L$ whose Gram matrix is $S_0$. The Minkowski-Siegel Reduction Algorithm, applied to the Gram matrix $S_0$, produces a matrix $U$ of $GL(4, \mathbb{Z})$ which can be used to define the four elements 
        \begin{equation}
\left\{ \begin{array}{llllllll}
\omega_{1}=& n_{11}v_1 & + & n_{12} v_2 & + & n_{13} v_3 & + &n_{14} v_4\\
 \omega_{2}= & n_{21}v_1 & + & n_{22} v_{2} & + & n_{23} v_{3} & +
	 &n_{24} v_{4} \\
         \omega_{3} = & n_{31}v_1 & + & n_{32} v_{2} & + & n_{33} v_{3} & +
	 &n_{34} v_{4} \\ 
	 \omega_{4} = & n_{41}v_1 & + & n_{42} v_{2} & + & n_{43} v_{3} & +
	 &n_{44} v_{4} .
	 \end{array} \right.
	 \end{equation}
   The elements $\{\omega_{1}, \omega_{2},\omega_{3}, \omega_{4}\}$
   form a basis of $L$ since the matrix $U$ 
   belongs to $GL(4, \mathbb{Z})$, with its inverse. 
 Notice that 
    \begin{equation*}
    U
\left( \begin{array}{l}
          v_1 \\
	  v_{2} \\
         v_{3} \\ 
	 v_{4}
	 \end{array} \right) =
\left( \begin{array}{l}
    \omega_1 \\
	  \omega_{2} \\
         \omega_{3} \\ 
	 \omega_{4}
	 \end{array} \right)
 \end{equation*}
and therefore that the two bases $\{  v_1, v_{2}, v_{3}, v_{4}\}$ and $\{  \omega_1, \omega_{2}, \omega_{3}, \omega_{4}\}$ are equivalent (in particular they generate the same lattice).
We conclude the remark by pointing out that the Gram matrix $R$ associated to the basis
   $\{\omega_{1}, \omega_{2},\omega_{3}, \omega_{4}\}$ is obtained
   as (recall formula \eqref{scalare}): 
   $$
  US_{0}\ ^{t}U= R=\left( \begin{array}{llll}
          \langle \omega_{1}, \omega_{1} \rangle & \langle \omega_{1}, \omega_{2} \rangle &\langle \omega_{1},
	  \omega_{3} \rangle & \langle \omega_{1}, \omega_{4} \rangle \\
	  \langle \omega_{2}, \omega_{1} \rangle & \langle \omega_{2}, \omega_{2} \rangle & \langle
	  \omega_{2}, \omega_{3} \rangle &
	 \langle \omega_{2}, \omega_{4} \rangle \\
         \langle \omega_{3}, \omega_{1} \rangle &\langle \omega_{3}, \omega_{2} \rangle & \langle
	 \omega_{3}, \omega_{3} \rangle & \langle \omega_{3}, \omega_{4} \rangle \\ 
	 \langle \omega_{4}, \omega_{1} \rangle & \langle \omega_{4}, \omega_{2} \rangle & \langle
	 \omega_{4}, \omega_{3} \rangle &\langle \omega_{4}, \omega_{4} \rangle
	 \end{array} \right) 
   $$
   and is therefore independent of the choice of a particular basis $\{  v_1, v_{2}, v_{3}, v_{4}\}$ among those that have the same Gram matrix $S_0$. In fact, in this sense,  the matrix $R$ depends only on the Gram matrix $S_0$.
   \end{oss}
 
To classify rank-$4$ lattices and generated tori, we will define the set of bases naturally emphasized by Algorithm \ref{kowski}.

\begin{definiz}  Let $L$ be a rank-$4$ lattice in $\mathbb{H}$. A basis $\{\omega_1, \omega_{2},\omega_{3}, \omega_{4}\}$ of $L$  will be called a \emph{reduced basis} if its Gram matrix is reduced.
 \end{definiz}

A direct application of the Minkowski-Siegel Reduction Algorithm and Remark \ref{chiave} prove a first result in the study of equivalence of lattices.

\begin{teo}\label{baseridotta}
Let $L$ be a rank-$4$ lattice and let $\{v_{1}, v_{2}, v_{3}, v_{4}\}$ be a basis of $L$. Then there exists a matrix $U$ of $GL(4, \mathbb{Z})$ such that
 \begin{equation*}
    U
\left( \begin{array}{l}
          v_1 \\
	  v_{2} \\
         v_{3} \\ 
	 v_{4}
	 \end{array} \right) =
\left( \begin{array}{l}
    \omega_1 \\
	  \omega_{2} \\
         \omega_{3} \\ 
	 \omega_{4}
	 \end{array} \right)
 \end{equation*}
is a reduced basis of the lattice $L$. As a consequence $L=\mathbb{Z}v_1+\mathbb{Z}v_2+\mathbb{Z}v_3+\mathbb{Z}v_4=\mathbb{Z}\omega_1+\mathbb{Z}\omega_2+\mathbb{Z}\omega_3+\mathbb{Z}\omega_4$ and we can suppose that the lattice $L$ is generated by a reduced basis.
\end{teo}

We will now present the basic features of reduced Gram matrices (and bases). 
\begin{prop} \label{condizio1/2}
If $R=(r_{i,j})_{i,j=1, \cdots 4} $ is a reduced Gram matrix, then the two following conditions hold:
\begin{enumerate}
\item $r_{k,k} \leq r_{l,l}$, for all $(k,l)\in \{1,2,3,4\} \times \{1,2,3,4\}$ such that $l>k$;
\item $-\frac{1}{2} r_{l,l} \leq r_{k,l} \leq \frac{1}{2} r_{l,l}$, for all $(k,l)\in \{1,2,3,4\} \times \{1,2,3,4\}$ such that $l\neq k$.
\end{enumerate}
\end{prop} 
\begin{proof} To verify $(1)$, we use condition B2) at step $k$, applied to the vector $e_l=(n_1,n_2,n_3,n_4)$ ($l>k$)  where $e_l$ is the $l$-th vector  of the standard basis of $\mathbb{R}^4$. Inequalities $(2)$ are obtained by applying the same condition B2) at step $k$, to the vector $(n_1,n_2,n_3,n_4)=e_k\pm e_l$ ($l\neq k$), where $e_l,e_k$ are, respectively, the $l$-th and $k$-th vector  of the standard basis of $\mathbb{R}^4$ (see also \cite[page 123]{maas}).
\end{proof}
\begin{oss} Concerning conditions B2) on the Gram matrix $R=(r_{i,j})_{i,j=1, \cdots 4}$, we observe that:  for each $k\in \{1,2,3,4\}$, if $e_k$ is the $k$-th element of the standard basis of $\mathbb{R}^4$,
 then we obtain the obvious equality $e_k R\ ^te_k=r_{k,k}$ that gives no conditions.
\end{oss}

We restate here a deep result that appears in \cite[theorem stated at page 139]{maas}, adapting it to our setting and notations.

\begin{teo} \label{generale} The set $\mathcal R$ of all reduced Gram matrices is a convex cone in $\mathbb R^{10}$ with $\mathring{\mathcal R} \neq \emptyset$. If $\mathcal G$ denotes the set of all Gram matrices, then 
\begin{equation}
\mathcal G= \bigcup_{U\in GL(4, \mathbb{Z})} U  \mathcal{R} \ ^tU.
\end{equation}
If $U\in GL(4, \mathbb{Z}), U\neq \pm I_4$, and $\mathcal R \cap (U \mathcal{R}\ ^tU) \neq \emptyset$, then $\mathcal R \cap (U \mathcal{R} ^tU) \subset \partial \mathcal R$. Only for a finite set of matrices $U\in GL(4, \mathbb{Z})$ it is possible that $\mathcal R \cap (U  \mathcal{R}  ^tU) \neq \emptyset$.
\end{teo}

The fact that $\mathring{\mathcal R} \neq \emptyset$ is not obvious, and a non constructive proof is given in \cite[page 137]{maas}. 

Now, our second step in the classification of  rank-$4$ lattices and quaternionic tori makes use of a proper subset of the set of reduced bases.

\begin{definiz}\label{special basis}
A reduced basis $\{\omega_{1}, \omega_{2}, \omega_{3}, \omega_{4}\}$ of a rank-$4$ lattice $L$ with the property that $\omega_1=1$ will be called a \emph{special basis}.    
\end{definiz}

\begin{teo}\label{basespeciale}
Let $L_1$ be a rank-$4$ lattice and let $\{v_{1}, v_{2}, v_{3}, v_{4}\}$ be a basis of $L_1$. Then $\{v_{1}, v_{2}, v_{3}, v_{4}\}$ is equivalent to a special basis $\{1, \omega_{2}, \omega_{3}, \omega_{4}\}$ of a rank-$4$ lattice $L_2$.
As a consequence $L_1=\mathbb{Z}v_1+\mathbb{Z}v_2+\mathbb{Z}v_3+\mathbb{Z}v_4$ and $L_2=\mathbb{Z}+\mathbb{Z}\omega_2+\mathbb{Z}\omega_3+\mathbb{Z}\omega_4$ are equivalent, and hence they generate equivalent  tori.
\end{teo}

\begin{proof} By Theorem \ref{baseridotta}, there
there exist a matrix $U$ of $GL(4, \mathbb{Z})$ and a reduced basis  $\{u_1, u_{2}, u_{3}, u_{4}\}$ of the lattice $L_1$ such that
 \begin{equation*}
    U
\left( \begin{array}{l}
          v_1 \\
	  v_{2} \\
         v_{3} \\ 
	 v_{4}
	 \end{array} \right) =
\left( \begin{array}{l}
    u_1 \\
	  u_{2} \\
         u_{3} \\ 
	 u_{4}
	 \end{array} \right).
 \end{equation*}
 Therefore
  \begin{equation*}
    U
\left( \begin{array}{l}
          v_1 \\
	  v_{2} \\
         v_{3} \\ 
	 v_{4}
	 \end{array} \right) =
\left( \begin{array}{l}
    1 \\
	  u_{2}u_1^{-1} \\
         u_{3}u_1^{-1} \\ 
	 u_{4}u_1^{-1}
	 \end{array} \right)u_1=
	 \left( \begin{array}{l}
    1 \\
	  \omega_{2} \\
         \omega_{3} \\ 
	 \omega_{4}
	 \end{array} \right)u_1
 \end{equation*}
and, by Definition \ref{equivlattices}, the bases $\{v_{1}, v_{2}, v_{3}, v_{4}\}$ and  $\{1, \omega_{2}, \omega_{3}, \omega_{4}\}$ are equivalent. This latter basis is  special, since it is obtained multiplying the reduced basis $\{u_1, u_{2}, u_{3}, u_{4}\}$ on the right by $u_1^{-1}\neq 0$, which corresponds to applying a rigid motion composed with a ``positive'' homothety of $\mathbb{H}\cong \mathbb{R}^4$, see Remark \ref{omotetia}. Indeed such a transformation maintains the fact that the Gram matrix is reduced (see Proposition \ref{cnsfundaregion}).
\end{proof}
At this point it is possible to associate to each class of equivalence of quaternionic tori at least a special basis of a rank-$4$ lattice, according to

\begin{corollario}\label{i)ii)}  Let $T$ be a 
quaternionic torus. Then, up to biregular diffeomorphisms, we can suppose that $T=\mathbb{H}/L$  where the lattice $L$ is generated by a special basis $\{1, \omega_{2}, \omega_{3}, \omega_{4}\}$.
\end{corollario}


We will now pass to identify 
 a natural and useful subset of the possible bases for rank-$4$ lattices.
Let $p=x_0+x_1i+x_2j+x_3k  \in \mathbb{H}$ and let $A$ be a $4\times 4$ matrix with real coefficients. We set the notation $A(p)$ to denote the quaternion whose real components are
    \begin{equation}\label{matricevettore}
   A \left( \begin{array}{l}
          x_0 \\
	x_{1} \\
          x_{2} \\ 
	 x_{3}
	 \end{array} \right).  
\end{equation}
The following is a useful elementary result of linear algebra:
\begin{prop}\label{ortogonale}
If two bases $\{\omega_1, \omega_{2},\omega_{3}, \omega_{4}\}$ and $\{v_1, v_{2},v_{3}, v_{4}\}$ have the same Gram matrix, then there exists an orthogonal matrix $B \in O(4,\mathbb{R})$ such that $B(\omega_l)=v_l$ for $l=1,2,3,4.$ 
\end{prop} 
\begin{proof}
Since the two bases have the same Gram matrix we have $\langle \omega_l,\omega_p \rangle = \langle v_l,v_p \rangle$ for $l, p=1,...,4$. Let $B$ be the matrix which transforms the first basis into the second one, then $\langle \omega_l,\omega_p \rangle= \langle B(\omega_l), B(\omega_p) \rangle$, for all $l, p = 1, ...,4$.  Hence $B$ is an isometry with respect to the standard scalar product and the assertion follows. 
\end{proof} 

We will end this section with some remarks. A lattice $L$ is called \emph{normalized} if $1\in L$ and if every element of $L$ has norm greater or equal than $1$. 
It can be proved that conditions (1)-(2) of Proposition \ref{condizio1/2} together with B1) are sufficient for a Gram matrix to be a reduced Gram matrix if and only if the associated lattice is normalized. To clarify what we mean, we provide an example of a non-normalized lattice $L$ having a basis $\mathcal{B}$ whose Gram matrix  satisfies conditions (1)-(2) of Proposition \ref{condizio1/2} together with B1), but is not reduced. In fact $L$ is such that an integer combination of three vectors of  $\mathcal B$ is inside $\mathbb{B}.$ To see this, notice that 
 if $I=\frac{1}{\sqrt{2}}i + \frac{1}{\sqrt{2}}j$ and $J=\frac{1}{\sqrt{3}}i + \sqrt{\frac23}j$, then $\mathcal{B}= \{ 1, v_1=\,\,  e^{\frac{\pi}{3}I}, v_2=\,\, e^{\frac{2\pi }{3}J}, k \}$ is a basis for $L$ whose (approximated) Gram matrix 
$$
G=\left( \begin{array}{lllll}
1 & \frac{1}{2} & -\frac{1}{2} & 0 \\
\frac{1}{2} & 1 & (0.4891) & 0 \\
-\frac{1}{2} & (0.4891) & 1 & 0 \\
0 & 0 & 0& 1 
\end{array} \right) 
$$
satisfies conditions (1)-(2) of Proposition \ref{condizio1/2} and B1). Nevertheless it is easy to see that  $(1- v_1 + v_2) \in \mathbb{B},$ and hence $G$ is not reduced.

\section{A moduli space for quaternionic tori. Tame tori}

The aim of this section is to find a fundamental set, and possibly a moduli space, to ``parameterize"
the equivalence classes of quaternionic tori, with respect to the
action of biregular diffeomorphisms. We will then study the families of \emph{tame} lattices and \emph{tame} tori, whose definition is inspired by Theorem \ref{generale}, and whose moduli correspond to the interior of the fundamental set $\mathcal M$.

We will start by identifying a useful subset of the set $\mathcal R$ of reduced Gram matrices.

\begin{oss} In the sequel we will always consider reduced Gram matrices  associated to special bases, i.e. matrices of the form
 \begin{equation}
	S_{0}= \left( \begin{array}{llll}
         1 & \langle 1,  v_{2} \rangle &\langle 1,
	  v_{3} \rangle & \langle 1, v_{4} \rangle \\
	  \langle v_{2}, 1 \rangle & \langle v_{2}, v_{2} \rangle & \langle
	  v_{2}, v_{3} \rangle &
	 \langle v_{2}, v_{4} \rangle \\
         \langle v_{3}, 1 \rangle &\langle v_{3}, v_{2} \rangle & \langle
	 v_{3}, v_{3} \rangle & \langle v_{3}, v_{4} \rangle \\ 
	 \langle v_{4}, 1\rangle & \langle v_{4}, v_{2} \rangle & \langle
	 v_{4}, v_{3} \rangle &\langle v_{4}, v_{4} \rangle
	 \end{array} \right) 
\end{equation}
where $\langle v_{4}, v_{4} \rangle\geq \langle v_{3}, v_{3} \rangle\geq \langle v_{2}, v_{2} \rangle\geq 1$. This means, in particular,  that we restrict to reduced Gram matrices belonging to an affine hyperplane of $\mathbb{R}^{10}$. We point out that we will consider, in the boundary of the set of reduced Gram matrices, only those elements that represent rank-$4$ lattices, and hence only definite positive matrices. Instead, as it appears in \cite[page 136]{maas}, 
when considering the entire set of reduced Gram matrices as a subset of the space of symmetric matrices, then its boundary contains also semi-positive definite, reduced 
matrices.
\end{oss}

The promised space of  ``parameters" for the equivalence classes  of biregular diffeomorphism of quaternionic tori is defined as follows. 
\begin{definiz}
The set $\mathcal M$ defined as
\begin{equation}\label{esse}
\mathcal{M}=\{ (v_2,v_3,v_4)\in \mathbb{H}^3: \{1, v_2,v_3,v_4 \} \hbox{ is a special basis} \}
\end{equation}
is called the \emph{fundamental set of quaternionic tori}. If we identify elements of $\mathcal M$ when they correspond to special bases originating regularly diffeomorphic tori, then as customary the quotient set  $\widetilde{\mathcal M}$ is called \emph{moduli space of quaternionic tori.} Let $(v_2,v_3,v_4)$ be a point of $\mathcal M$, and let $L=\mathbb{Z}+\mathbb{Z}v_2+\mathbb{Z}v_3+\mathbb{Z}v_4$ be the lattice generated by the special basis $\{1, v_2,v_3,v_4 \}$. We will say that $(v_2,v_3,v_4)$ is  \emph{(a representative of) the modulus} of any quaternionic torus equivalent to $T=\mathbb{H}/L$.
\end{definiz}
Corollary \ref{i)ii)} guarantees an important property of the fundamental set:
\begin{prop}
Every quaternionic torus $T$ has (at least) a modulus in $\mathcal{M}$. In other words: for every quaternionic torus $T$, there exists $(v_2, v_3, v_4)\in \mathcal M$ such that $T$ is equivalent to $\mathbb{H}/L$,  where $L=\mathbb{Z}+\mathbb{Z}v_2+\mathbb{Z}v_3+\mathbb{Z}v_4$ is the lattice generated by the special basis $ \{1, v_2,v_3,v_4 \} $. 
\end{prop}
  With obvious notations, set now
$$
\hat O(3,\mathbb{R})=\begin{bmatrix}
1& \underline 0\\
^t\underline 0& O(3, \mathbb{R})
\end{bmatrix}
$$
and define 
\begin{equation}\label{emme}
\mathcal S = \bigcup_{B\in \hat O(3,\mathbb{R})} B. \mathcal M = \hat O(3, \mathbb{R}). \mathcal M
\end{equation}
where $B. \mathcal M$ means the set of all $B.(v_2,v_3,v_4)= (B(v_2), B(v_3), B(v_4))$, for all $(v_2,v_3,v_4) \in \mathcal M.$ The fundamental set $\mathcal M$  has a natural symmetry, namely we have that
\begin{prop}\label{biri2}
The fundamental set  $\mathcal M$ and the set $\mathcal S$ coincide. Equivalently, $\hat O(3,\mathbb{R})$ acts on $\mathcal M$. Moreover, if $(v_2, v_3, v_4)\in \mathcal M$, then all elements  of its orbit with respect to the action of $\hat O(3,\mathbb{R})$, 
\[
\hat O(3,\mathbb{R}).(v_2, v_3, v_4),
\]
correspond to special bases having the same reduced Gram matrix as $\{1, v_2, v_3, v_4\}$.
\end{prop}
\begin{proof} The proof is a direct computation.
\end{proof}

The geometric symmetry of the fundamental set stated in Proposition \ref{biri2} is interesting, and suggests a remark and a few considerations, that help to identify similarities in the moduli of tori.

\begin{oss}\label{biri}
Denote as usual by $\mathcal E =\{1, i, j, k\}$ the standard basis for the space $\mathbb{H}$ of quaternions. Recall that, for every two unitary quaternions $I, J\in \mathbb{S}$ with $I\perp J$, the set $\mathcal A=\{1, I, J, IJ=K\}$ is also a (positively oriented) basis for the space $\mathbb{H}$, having the same multiplication rules of the basis $\mathcal E$.  Let us consider two lattices $L_1$ and $L_2$ generated, respectively, by the special bases $\mathcal V=\{1, v_2, v_3, v_4\}$ and $\mathcal W=\{1, w_2, w_3, w_4\}$. Let the coefficients of the elements of $\mathcal V$ with respect to the basis $\mathcal E$ coincide with the coefficients of the elements of $\mathcal W$ with respect to the basis $\mathcal A$. Then, in view of the last statement of Proposition \ref{biri2}, the two generated tori $T_1$ and $T_2$ - notwithstanding not equivalent according to our definition - have the same reduced Gram matrix and completely similar structures.
\end{oss}
Let $L$ be a lattice  having a special basis $\mathcal V=\{1, v_2, v_3, v_4\}$. Then there exists a basis $\mathcal A=\{1, I, J, IJ=K\}$ of $\mathbb{H}$ such that $v_2 \in \mathbb{R} + I\mathbb{R}$. What is stated in  Proposition \ref{biri2} and in Remark \ref{biri} allows us to study only the case of lattices having a special basis $\mathcal V=\{1, v_2, v_3, v_4\}$ with $v_2 \in \mathbb{R} + I\mathbb{R}$, where $I\in \mathbb S$ is the second element of a (positively oriented) basis  $\mathcal A=\{1, I, J, IJ=K\}$ of $\mathbb H$.
\vskip .3cm
It is easy to see (and in any case we will see it later on, in this paper) that there are different elements belonging to $\partial \mathcal M$ that correspond to the same equivalence class of quaternionic tori,  or equivalently that there are quaternionic tori having more than one representative in $\mathcal M$. However, this last phenomenon is not present in the case of the  family of quaternionic  tori that we are going to define.

\begin{definiz}\label{tameness} 
 Let $L$ be a rank-$4$ lattice in $\mathbb{H}$. 
 \begin{enumerate}
 \item The lattice $L$ is called a \emph{tame lattice} if there exists a reduced basis of $L$ 
whose Gram matrix is an interior point of $\mathcal R$. Such a basis will be called a \emph{tame basis}.
\item A quaternionic torus $T$ is called a \emph{tame torus} if there exists a tame lattice $L$ such that $T$ is equivalent to $\mathbb{H}/L$. 
 \end{enumerate}
 \end{definiz}

%

\noindent Here is an easy criterion to decide if a given torus is tame or not.

\begin{prop} \label{baseunica}
 Let $L$ be a lattice and $\mathcal B = \{ v_1, v_2,v_3,v_4 \}$ be a reduced basis for $L$. Consider the torus $T = \mathbb{H}/L$.
Then $T$ is a tame torus, if, and only if, $\pm\mathcal{B}$ are the unique reduced bases for $L$.

\end{prop}
\begin{proof}
If the torus $T$ is tame, then suppose  that there are two different reduced bases, $\mathcal{B} = \{ v_1, v_2,v_3,v_4 \}$ and $\mathcal{B}_1=\{\omega_1, \omega_{2},\omega_{3}, \omega_{4}\}$ for the tame lattice $L$, with $\mathcal{B}_1\neq \pm\mathcal{B}$. 
As a consequence, there exists $U\in GL(4,\mathbb{Z})\setminus \{\pm I_4\}$ such that 
 \begin{equation*}
    U
\left( \begin{array}{l}
          v_1 \\
	  v_{2} \\
         v_{3} \\ 
	 v_{4}
	 \end{array} \right) =
\left( \begin{array}{l}
    \omega_1 \\
	  \omega_{2} \\
         \omega_{3} \\ 
	 \omega_{4}
	 \end{array} \right).
 \end{equation*}
If $R$ and $R_1$ denote the reduced Gram matrices associated, respectively,  to the bases $\mathcal{B}$ and $\mathcal{B}_1$, then 
$$
U\ R\ ^tU=R_1.
$$
Therefore, $R$ and $R_1$ belong to the boundary of $\mathcal{R}$ by Theorem \ref{generale}, and hence the torus is not tame.
To prove the converse, suppose that $T$ is not tame, i.e. that the Gram matrix $R=(r_{i,j})$ associated to $\mathcal{B}$ belongs to $\partial\mathcal{R}$. Therefore, equality holds either  in B1) or in B2). In the first case, since $r_{k,k+1}=0$ for some $k=1,2,3$, the two consecutive vectors $v_k$ and $v_{k+1}$ are orthogonal. We can then consider a second reduced basis $\mathcal{B'}$ obtained by substituting  $v_{k+1}$ with $-v_{k+1}$, ... , $v_{4}$ with $-v_{4}$. If instead equality holds in B2), 
the Minkowski-Siegel Reduction Algorithm directly implies the existence of a second reduced basis.

\end{proof}

We will now  find, inside the fundamental set  $\mathcal M$, a moduli space for the classes of equivalence of tame quaternionic tori. In fact if we set
\begin{equation}\label{tau}
\mathcal{T}=\{ (v_2,v_3,v_4)\in \mathcal M: \{1, v_2,v_3,v_4 \} \hbox{ is a (special) tame basis} \}
\end{equation}
then, with the aid of a preliminary lemma, we can prove a uniqueness result for the modulus of a tame torus.

\begin{lemma} \label{norma a}
Let two  lattices $L_1$ and $L_2$ of $\,\,\mathbb{H}$ be generated respectively, by the special bases $\{ 1, \alpha_2, \alpha_3, \alpha_4 \}$ and $\{ 1, \omega_2, \omega_3, \omega_4 \}.$ If $F(q)=q  a,$ with $a \in \mathbb{H}^*,$ is an automorphism of $\mathbb{H}$ such that $F(L_1)=L_2,$ then $|a|=1.$
\end{lemma}
\begin{proof} Since $a=F(1)$, we have that
$a \in L_2,$ and hence $|a| \ge 1.$ Moreover, since $\{a, \alpha_2 a, \alpha_3 a, \alpha_4 a \}$ are linearly independent vectors which generate the lattice  $L_2,$ then there exist $n_1, n_2, n_3, n_4 \in \mathbb{Z}$ such that $n_1 (a)+n_2 (\alpha_2 a) + n_3 (\alpha_3 a) +n_4 (\alpha_4 a) =1$. This equality implies that $(n_1 +n_2\alpha_2 + n_3 \alpha_3 + n_4 \alpha_4)a=1.$ But 
$|n_1 +n_2\alpha_2 + n_3 \alpha_3 + n_4 \alpha_4| \ge 1$ because $n_1 +n_2\alpha_2 + n_3 \alpha_3 + n_4 \alpha_4$ is an element of $L_1.$ Hence $|a| \le 1.$ 
\end{proof}

\begin{teo} The set $\mathcal T\subset \mathcal M$ is  a moduli space for the equivalence classes of tame tori. In other words, every tame torus has exactly one representative in $\mathcal T$ (its modulus).
\end{teo}
\begin{proof} 
Suppose that the two elements $V=(v_2,v_3,v_4)\in \mathcal T$ and $ W=(\omega_2,\omega_3,\omega_4)\in \mathcal{T}$ correspond to equivalent tame tori. If this is the case, then (see Definition \ref{equivlattices}) there exist $U\in GL(4,\mathbb{Z})$ and a quaternion $a\neq 0$ such that
 \begin{equation*}
    U
\left( \begin{array}{l}
          1 \\
	  \omega_{2} \\
          \omega_{3} \\ 
	  \omega_{4}
	 \end{array} \right) =
\left( \begin{array}{l}
    1 \\
	  v_{2} \\
      v_{3} \\ 
	 v_{4}
	 \end{array} \right) a.
 \end{equation*}
Lemma \ref{norma a} implies that $|a|=1$. Since multiplication by the unitary quaternion $a$ is an Euclidean rotation (see Remark \ref{omotetia}), then $(a, Va)$ has the same (reduced) Gram matrix as $(1, V)$, and hence it is a reduced basis.  Now, since $(a,Va)$ and $(1,W)$ are both  reduced bases, then by Proposition \ref{baseunica} we reach the conclusion that $a=\pm 1$ and hence that $V=W$.
\end{proof}

 \section{On the groups of automorphisms of ``boundary" tori} \label{TQM}

According to Theorem \ref{generale}, a reduced Gram matrix $R$ belongs to $\partial \mathcal R$ if, and only if, there exists a reduced Gram matrix $S$ such that 
$$
S=UR\ ^tU
$$
for some $U\in GL(4, \mathbb{Z})$, $U\neq \pm I_4$. Notice that to each of these reduced Gram matrices $R, S$ there correspond infinitely many $GL(4, \mathbb{Z})$-nonequivalent bases (see \eqref{emme}).
Therefore the study of the equivalence classes of non tame tori consists in the identification and classification of reduced Gram matrices belonging to the boundary of $\mathcal{R}$, and corresponding to non equivalent special bases. We plan to address this fascinating problem in a forthcoming paper. 

However, as it happens in the complex case, the first interesting and fundamental step in this direction is the search and classification of  boundary tori with non trivial groups of (biregular) automorphisms. These tori, which are the quaternionic counterpart of tori with complex multiplication (classically indicated as \emph{harmonic} and \emph{equianharmonic}), will be found and classified in the rest of this section.

\begin{oss}\label{pianiinvarianti}

Every vector of $\mathbb{H}\cong \mathbb{R}^4$ lies in an invariant real plane of the rotation $q\mapsto qa,$ (see Remark \ref{omotetia}), where 
$a=\cos \alpha +I_a \sin \alpha.$ This fact can be verified directly as follows: for any quaternion $q,$ we define $q'=qI_a,$ and notice that, since $1$ and $I_a$ are perpendicular vectors in $\mathbb{H} \cong \mathbb{R}^4,$ so are $q$ and $q'.$ Moreover  $I_a^2=-1,$ $q' I_a=-q;$ thus 
$$a=\cos \alpha +I_a \sin \alpha$$
$$q a= q\cos \alpha + q' \sin \alpha, \,\,\,\,\,\,\, q' a= q' \cos \alpha -q \sin \alpha. $$
Therefore, the plane containing the vectors $q$ and $q'$ is invariant, and the rotation in this plane is of an angle $\alpha,$ \cite[page 37]{DuV}. \end{oss} 

\begin{lemma} \label{norma a bis}
Let $L$ be a lattice of $\,\,\mathbb{H}$ generated by the special basis $\{ 1, \alpha_2, \alpha_3, \alpha_4 \}$ and let $F(q)=q a,$ with $a \in \mathbb{S}^3,$ be an automorphism of $\mathbb{H}$ such that $F(L)=L.$ Then $a$ has finite order, i.e. there exists $n \in \mathbb{N}$ such that $a^n=1,$ 
and the order of $a$ divides either $4$ or $6.$ 
\end{lemma}
\begin{proof}
Since $1$ is an element of the lattice $L$ and since $F(q)=q  a$ maps $L$ onto $L,$ it follows that $1a=a \in L;$ similarly for all $m \in \mathbb{N},$ it holds that $a^m \in L.$ By compactness, the sequence $\{a^m \}_{m \in \mathbb{N}}$ of unit vectors in $L,$ has a convergent subsequence.
Unless $\{a^m \}_{m \in \mathbb{N}}$ is a finite set, this is in contradiction with Lemma \ref{LemmaCauchy} and Theorem \ref{discreteness} which assert respectively that $L$ is a (closed) discrete subgroup of $\mathbb{H}$. In order to prove the second assertion, we use what is stated in Remark \ref{pianiinvarianti}:
since $1$ and $a$ are elements of $L,$ it follows that the complex plane $L_{I_a}$, which contains $1$ and $a,$ is invariant by right multiplication by $a,$ i.e. the integer combinations of $1$ and $a$ form a rank-$2$ sublattice $L'$ of $L,$ contained in the complex plane $L_{I_a}$, with $F(q)=q a$ as an automorphism restricted to the sublattice $L'\subset L_{I_a}$; therefore $a$ is a root of unity of order $n,$ where $n$ divides either $4$ or $6,$ because in the complex setting these are the only possibilities (see, e.g., \cite[page 82]{DuV}).
\end{proof}

Proposition \ref{lifting} directly suggests how to define the automorphisms of a quaternionic torus.

 \begin{definiz}
Let $T=\mathbb{H}/L$ be the quaternionic torus associated to the rank-$4$ lattice $L$. The group of \emph{biregular automorphisms of the torus $T$} is defined as 
$$Aut(T)=\{ F\in Aut(\mathbb{H}) \,\,| \,\, F(q)=q  a \,\,\, \textrm{with} \,\,\  a \in \mathbb{S}^3 \,\,\, \textrm{and} \,\,\, F(L)=L \}.$$
\end{definiz}
We point out that the group $Aut(T)$ of biregular automorphisms of the torus $T=\mathbb{H}/L$ can also be interpreted as the group of biregular automorphisms $Aut_0(L)$ of a rank-$4$ lattice $L$ fixing the point $0\in L\subset \mathbb{H}$.

\begin{prop} \label{permuta} Let $L$ be a rank-$4$ lattice of  $\mathbb H$ containing $1$. Let  $T=\mathbb{H}/L$ be  the associated  quaternionic torus, and let 
\[
A_T=\{ a \in \mathbb{S}^3 \,: \exists \ F\in Aut(T) \  \textnormal{defined as}\   F(q)=qa \}. 
\]
Then:
\begin{enumerate}
\item the set $A_T\subset \mathbb{S}^3\cap L$  is a subgroup (with respect to quaternionic multiplication) of the group $\mathbb{S}^3$ of unitary quaternions;
\item the group $A_T$ is isomorphic to $Aut(T)$;
\item  for any fixed $R\geq 0$, each 
$F \in Aut(T)$ acts as a permutation on the finite set of all vectors of $L\cap \partial B(0, R)$.
\end{enumerate}
\end{prop}
\begin{proof}
The group structure of the set $A_T$ with respect to the quaternionic multiplication is inherited by the one of $Aut(T)$ with respect to composition. The fact that each $F \in Aut(T)$ acts as a permutation on vectors of fixed norm is straightforward by Lemma \ref{norma a} and \ref{norma a bis} and by the fact that an automorphism $F$ of $T$ maps $T$ onto itself.
\end{proof}


We now pass to recall and list all the finite subgroups of the unitary quaternions, and refer the reader to the classical books \cite{conway}, \cite{Cox}, \cite{DuV} and  \cite{LT} for the underlying theory.
We point out that in the literature there is a diffused confusion in the use of notations that concern the groups we are dealing with; here we will mainly refer to, and use the notations of, the book  \cite{conway} by Conway and Smith.

We begin by listing all (up to conjugation) finite subgroups of the group of rotations $SO(3,\mathbb{R})$:
\begin{enumerate}
\item the \emph{icosahedral} group, \ \ $\mathbb{I} \cong A_5$, \ \ $60$ elements;
\item the \emph{octahedral} group, \ \ $\mathbb{O}\cong S_4$, \ \ $24$ elements;
\item the \emph{tetrahedral} group, \ \ $\mathbb{T} \cong A_4$, \ \ $12$ elements;
\item the \emph{dihedral} group, \ \ $D_{2n}\cong D(n)$, \ \ $2n$ elements;
\item the \emph{cyclic} group, \ \ $C_n \cong C(n)$, \ \ $n$ elements.
\end{enumerate}

Every unitary quaternion $q$ is associated to a precise rotation of $SO(3,\mathbb{R})$ by means of the $2-$to$-1$ correspondence that maps $q$ to the rotation $[q]: x \to \bar qxq$ (see, e.g., \cite[Theorem 4, page 24]{conway}). 
As a consequence, every finite group $Q$ of unitary quaternions is mapped to a group $[Q]=\{[q]: q\in Q\}$(isomorphic to) $C_n, D_{2n}, \mathbb{T}, \mathbb{O}, \mathbb{I}$. The number of elements of $Q$ is $2$ or $1$ times the number of elements of $[Q]$, according to whether $-1$ is or is not in $Q$. 

If $G$ denotes one of the finite subgroups of the group of rotations, then we set 
$$
2G=\{q \in \mathbb{S}^3: [q]\in G\}.
$$
The only possible cases in which $-1\in Q$ are those where $Q=2C_n, 2D_{2n}, 2\mathbb{T}, 2\mathbb{O}, 2\mathbb{I}$. On the other hand, let us suppose that $-1\not\in Q$. In this case, $G$ can contain no order $2$ rotation $g$: if $[q]=g$ then $q^2=-1$ must be in $Q$. The only group $G$ without order $2$ elements is $C_n$ with $n$ odd; this gives rise to a group $Q=1C_n$ in $\mathbb{S}^3$ isomorphic to $C_n$. In fact, the following result holds (see, e.g., \cite[Theorem 12, page 33]{conway}).

\begin{teo}
The finite subgroups of unitary quaternions are
$$
2\mathbb{I},\ \ \  2\mathbb{O},\ \ \  2\mathbb{T}, \ \ \  2D_{2n}, \ \ \  2C_n, \ \ \  1C_n \textnormal{($n$ odd)}.
$$
\end{teo}
\noindent With the usual notations for quaternions, let  $I,J\in \mathbb{S}$, with $I\perp J$, and let $\{1, I, J, IJ=K\}$ be a basis for $\mathbb{H}$ having the usual multiplication rules. We then set
\begin{align*}
&I_{\mathbb{I}}=\frac{I+\sigma J+\tau K}{2},\ \ \  \sigma=\frac{\sqrt 5-1}{2},\ \ \  \tau=\frac{\sqrt 5+1}{2};\\
&I_{\mathbb{O}}=\frac{J+K}{\sqrt 2};\\
&\omega=\frac{-1+I+J+K}{2};\\
&I_{\mathbb{T}}=I;\\
&e_n=e^{\frac{\pi I}{n}}.
\end{align*}

\begin{teo}\label{generazione}
The finite subgroups of unitary quaternions are generated as follows:
\begin{align*}
& 2\mathbb{I}=\langle\langle I_{\mathbb{I}}, \omega\rangle\rangle,\ \ \ 
2\mathbb{O}=\langle\langle I_{\mathbb{O}}, \omega\rangle\rangle,\ \ \ 
2\mathbb{T}=\langle\langle I_{\mathbb{T}}, \omega\rangle\rangle,\\
&2D_{2n}= \langle\langle e_n, j\rangle\rangle,\ \ \ 
2C_n=\langle\langle e_n\rangle\rangle,\ \ \  
1C_n=\langle\langle e_{\frac{n}{2}}\rangle\rangle \ \textnormal{($n$ odd)}.
\end{align*}
\end{teo}

\begin{teo}
There is no quaternionic torus whose group of automorphisms is isomorphic to $2\mathbb{I},\ 2\mathbb{O},\  2D_{2n} \textnormal{($n\ge 4$),} \  2C_n \textnormal{($n\ge 4$)}, \ 1C_n \textnormal{($n$ odd)}$.
\end{teo}
\begin{proof}
Since the subgroup $\mathbb{I}$ contains an element of order $5$, then $2\mathbb{I}$ has an element of order $10$. Hence, by Lemma \ref{norma a bis}, it cannot be the group of automorphisms of a quaternionic torus. The same argument holds to exclude $\ 2\mathbb{O}$ that has an element of order $8$. Analogously the groups $2D_{2n} \textnormal{($n\ge 4$),} \  2C_n \textnormal{($n\ge 4$)}$ are excluded since they both contain the element $e_n$ whose order is $2n$.
Finally, the group $1C_n \textnormal{($n$ odd)}$ cannot be isomorphic to the group of automorphisms of a quaternonic torus since it does not contain $-1$.
\end{proof}

This last result reduces the possible groups of automorphisms of a quaternionic torus to the list
$$
2\mathbb{T}, \ \ \ 2D_{2}, \ \ \ 2D_{4}, \ \ \ 2D_{6}, \ \ \ 2C_1, \ \ \ 2C_2, \ \ \ 2C_3.
$$
Since, as it is well known and easy to check, the groups $2C_2$ and $2D_2$ are isomorphic, the final list of these groups becomes
\begin{equation} \label{subgroupofS3}
2\mathbb{T}, \ \ \ 2D_{4}, \ \ \ 2D_{6}, \ \ \ 2C_1, \ \ \ 2C_2, \ \ \ 2C_3.
\end{equation}
\begin{oss}\label{inclusions}
We observe that the following group inclusions hold:
$$
2C_1\subset 2C_2 \subset 2D_{4} \subset 2\mathbb{T},
$$

$$
2C_1\subset 2C_3 \subset 2D_6,
$$
and
$$
2C_3 \subset 2\mathbb{T}.
$$
\end{oss}
For each group $2G$ in the list \eqref{subgroupofS3}, we will exhibit  all tori whose group of automorphisms contains, or coincides with, $2G$.  We will begin with the group $2C_1=\{1, -1\}$, which appears for each quaternionic torus. As the reader may imagine, for reasons of neat presentation we will from now on suppose, without loss of generality, that the lattices which generate the tori involved have $1$ as a vector of minimum modulus.
\begin{prop}
The group $2C_1$ is (isomorphic to) a subgroup of the group of biregular automorphisms of any quaternionic torus $T$.
\end{prop}
\begin{proof}Let $L$ be a rank-$4$ lattice of $\,\,\mathbb{H}$ generated by the special basis $\{ 1, \alpha_2, \alpha_3, \alpha_4 \}$ and such that $T= \mathbb{H}/ L$. Then $2C_1= \{1, -1\}$ consists of automorphisms of $T$ since $\{ -1, -\alpha_2, -\alpha_3, -\alpha_4 \}$ generates the lattice $L$.
\end{proof}

\begin{prop}\label{2C_2} 
 Let $T$ be a quaternionic torus. 
The group of biregular automorphisms of the torus $T$ contains a subgroup (isomorphic to) $2C_2\cong 2D_2$, if and only if 
there exists $I\in \mathbb{S}$ and a quaternion $\alpha_3$ with $|\alpha_3|\geq 1$ such that $(I, \alpha_3, \alpha_3 I )$ is (a representative of) the modulus of $T$.
\end{prop}
 \begin{proof} 
If there exists $I\in \mathbb{S}$ and $\alpha_3\in \mathbb{H}$ with $|\alpha_3|\ge 1$  such that $(I, \alpha_3, \alpha_3 I )$ is  (a representative of) the modulus of $T$, then $\mathcal B =\{ 1, I, \alpha_3, \alpha_3 I \}$ 
 is a special basis of a lattice $L$ such that $T$ is equivalent to $\mathbb{H}/L$. Then the four bases
\begin{eqnarray*}
\mathcal B I^0&=&\mathcal B =\{ 1, I, \alpha_3, \alpha_3 I \}\\
\mathcal B I ^1&=&\mathcal B I=\{ I, -1, \alpha_3 I,- \alpha_3 \}\\
\mathcal B I^2 &=&-\mathcal B =\{ -1, -I, -\alpha_3 ,-\alpha_3 I \}\\
\mathcal B I^3 &=&-\mathcal B I=\{ -I, 1,  -\alpha_3 I,\alpha_3  \}
\end{eqnarray*}
all generate the lattice $L$, and hence the subgroup $2C_2\cong \{\pm 1, \pm I \}$ is a subgroup of the group of automorphisms $Aut(T)$. On the other hand, if $2C_2\subseteq Aut(T)$, then thanks to Theorem \ref{generazione}, Lemma \ref{norma a bis} and Proposition \ref{permuta},  the vectors $\{\pm 1, \pm I \}$ must belong to, and generate, the rank-$2$ sublattice $L\cap L_I$ of $L$. Therefore (by the classification of rank-$2$ lattices of $\mathbb{C}$), while using the Minkowski-Siegel Reduction Algorithm, we can choose  $\alpha_2=I$ in the special basis  $\mathcal B =\{ 1, \alpha_2, \alpha_3, \alpha_4 \}$ that generates $L$. We can also suppose that the third vector $\alpha_3$ is chosen (again according to  Algorithm \ref{kowski}) among those vectors that can complete the special basis $\mathcal{B}$. As a consequence, $\alpha_3 I, -\alpha_3, -\alpha_3 I, $ must belong to $L$ together with $\alpha_3$. Since these three elements of $L$ have the same norm of $\alpha_3$,  
and since $Aut(T)$ contains a subgroup isomorphic to $2C_2$, then the Minkowski-Siegel Reduction Algorithm \ref{kowski} can produce a special basis that generates $L$ by using suitable $\alpha_3$ and taking, automatically, $\alpha_4=\alpha_3 I$;  this concludes the proof. 
\end{proof}

Let $\{1, I, J, K\}$ be a standard basis for the skew field of quaternions. We recall that the ring of \emph{Lipschitz quaternions}  (or \emph{Lipschitz integers}) consists of the set $\mathcal L=\{m+nI+pJ+qK : m,n,p,q \in \mathbb{Z}\}\subset \mathbb{H}$. The ring $\mathcal L$ is, in turn, a subring of the ring of \emph{Hurwitz quaternions} (or \emph{Hurwitz integers}) $\mathcal{H}=\{ a+bI+cJ+dK \,\, : \,\, a,b,c,d \in \mathbb{Z} \ \  \textnormal{or}\ \  a,b,c,d \in \mathbb{Z}+\frac12\}$. The surprising properties of these rings are described, for instance, in \cite[Chapter II, Section 5]{conway}.

\begin{oss}
Concerning the proof of Proposition \ref{2C_2}, notice that only in the case in which the lattice $L$ consists of the ring of Hurwitz integers $\mathcal{H}$, it can happen that there exists $\alpha_3$ such that the special basis $\mathcal B =\{ 1, I, \alpha_3, \alpha_3 I \}$ simply generates a proper sublattice of $L$ and not the whole $L:$ for example if $\alpha_3=J$ is an imaginary unit quaternion orthogonal to $I,$ then the set $\mathcal B =\{ 1, I, J, JI \}$ generates the sublattice of the Lipschitz quaternions instead of the whole lattice of Hurwitz quaternions. According to Remark \ref{inclusions} and to the classification \eqref{subgroupofS3} of the finite subgroups of unitary quaternions which can be contained in $L \cap \mathbb{S}^3$, this is the only case in which  a set of linearly independent vectors of type $\mathcal B =\{ 1, I, \alpha_3, \alpha_3 I \}\subset L$ (with $|\alpha_3|= 1$) can generate a proper sublattice instead of the whole lattice. In this particular case, it is enough to change $\alpha_3$ with another vector of $L \cap \mathbb{S}^3$ which can be reached by means of  the Minkowski-Siegel Reduction Algorithm and such that $1,I,\alpha_3$ are not in the same multiplicative subgroup of $L \cap \mathbb{S}^3:$ we know that at least an $\alpha_3$ of this kind exists (this last fact depends on the well known structure of the subgroups of $2 \mathbb{T}$).
\end{oss}

If the group  of automorphisms $Aut(T)$  of the torus $T$ contains a subgroup isomorphic to $2C_2$, and if the torus $T$ has a special basis  of type $\{1, I, \alpha_3, \alpha_3I\}$ with  $|\alpha _3|>1$, then $Aut(T)\cong 2C_2$: this is a consequence of the classification of the rank-$2$ lattices of $\mathbb{C}$, and of the fact that, in these hypotheses, there are only four points in $L\cap \mathbb{S}^3$, all belonging to $L_I\cap \mathbb{S}^3$ (see Proposition \ref{permuta}).

\begin{definiz}
A quaternionic torus whose group of biregular automorphisms is isomorphic to $2C_2\cong 2D_2$ is called a \emph{cyclic-dihedral torus}.
\end{definiz}

\begin{prop}\label{2C_3} Let $T$ be a quaternionic torus. 
The group of biregular automorphisms of the torus $T$ contains a subgroup (isomorphic to) $2C_3$ if and only if there exists $I\in \mathbb{S}$ and a quaternion $\alpha_3$ with $|\alpha_3|\geq 1$ such that $(e^{\frac{\pi I}{3}}, \alpha_3, \alpha_3 e^{\frac{\pi I}{3}})$ is (a representative of) the modulus of $T$.
\end{prop}
\begin{proof} 
If $(e^{\frac{\pi I}{3}}, \alpha_3, \alpha_3 e^{\frac{\pi I}{3}})$ is  (a representative of) the modulus of $T$ then $\mathcal B =\{ 1, e^{\frac{\pi I}{3}}, \alpha_3, \alpha_3 e^{\frac{\pi I}{3}} \}$ is a special basis of a lattice $L$ such that $T$ is equivalent to $\mathbb{H}/L$. Then the six bases
\begin{eqnarray*}
\mathcal B (e^{\frac{\pi I}{3}})^0 &=&\mathcal B =\{ 1, e^{\frac{\pi I}{3}}, \alpha_3, \alpha_3 e^{\frac{\pi I}{3}} \}\\
\mathcal B (e^{\frac{\pi I}{3}})^1 &=&\mathcal B e^{\frac{\pi I}{3}}=\{ e^{\frac{\pi I}{3}}, e^{\frac{2\pi I}{3}}, \alpha_3 e^{\frac{\pi I}{3}}, \alpha_3 e^{\frac{2\pi I}{3}} \}\\
\mathcal B (e^{\frac{\pi I}{3}})^2 &=&\mathcal B e^{\frac{2\pi I}{3}} =\{ e^{\frac{2\pi I}{3}}, -1, \alpha_3e^{\frac{2\pi I}{3}}, -\alpha_3 \}\\
\mathcal B (e^{\frac{\pi I}{3}})^3 &=&-\mathcal B =\{ -1, -e^{\frac{\pi I}{3}}, -\alpha_3, -\alpha_3 e^{\frac{\pi I}{3}} \}\\
\mathcal B (e^{\frac{\pi I}{3}})^4 &=&-\mathcal B e^{\frac{\pi I}{3}}=\{ -e^{\frac{\pi I}{3}}, -e^{\frac{2\pi I}{3}}, -\alpha_3 e^{\frac{\pi I}{3}}, -\alpha_3 e^{\frac{2\pi I}{3}} \}\\
\mathcal B (e^{\frac{\pi I}{3}})^5 &=&-\mathcal B e^{\frac{2\pi I}{3}}=\{ -e^{\frac{2\pi I}{3}}, 1, -\alpha_3e^{\frac{2\pi I}{3}}, \alpha_3 \}
\end{eqnarray*}
all generate the lattice $L$, and hence the subgroup $2C_3\cong \{\pm 1, \pm  e^{\frac{\pi I}{3}}, \pm e^{\frac{2\pi I}{3}}\}$ is a subgroup of the group of automorphisms $Aut(T)$. On the other hand, if $2C_3\subseteq Aut(T)$, then thanks to Theorem \ref{generazione}, Lemma \ref{norma a bis} and Proposition \ref{permuta}, there exists $I\in \mathbb S$ such that the vectors $\{\pm 1, \pm  e^{\frac{\pi I}{3}}, \pm e^{\frac{2\pi I}{3}}\}$ must belong to, and generate, the rank-$2$ sublattice $L\cap L_I$ of $L$. Therefore (using the classification of the rank-$2$ lattices of $\mathbb{C}$),  while using the Minkowski-Siegel Reduction Algorithm to construct  the special basis  $\mathcal B =\{ 1, \alpha_2, \alpha_3, \alpha_4 \}$ that generates $L$,  we can choose  $\alpha_2=e^{\frac{\pi I}{3}}$. 
We can also suppose that the third vector $\alpha_3$ is chosen (according to  Algorithm \ref{kowski}) among those vectors that can complete the special basis $\mathcal{B}$. As a consequence, the points $\alpha_3 e^{\frac{\pi I}{3}},  \alpha_3e^{\frac{2\pi I}{3}}, -\alpha_3, -\alpha_3 e^{\frac{\pi I}{3}}, -\alpha_3e^{\frac{2\pi I}{3}} $ must belong to $L$ together with $\alpha_3$. Since these five elements have the same norm of $\alpha_3$, 
and since $Aut(T)$ contains a subgroup isomorphic to $2C_3$, then the Minkowski-Siegel Reduction Algorithm \ref{kowski} can produce a special basis that generates $L$ by using suitable $\alpha_3$ and, automatically, $\alpha_4=\alpha_3e^{\frac{\pi I}{3}}$; this completes the proof.
\end{proof}

If the group  of automorphisms $Aut(T)$  of the torus $T$ contains a subgroup isomorphic to $2C_3$, and if the torus $T$ has a 
special basis  of type $\{1, e^{\frac{\pi I}{3}}, \alpha_3, \alpha_3 e^{\frac{\pi I}{3}} \}$ with  $|\alpha _3|>1$, then $Aut(T)\cong 2C_3$: this is a consequence of the classification of the rank-$2$ lattices of $\mathbb{C}$, and of the fact that, in these hypotheses, there are only six points in $L\cap \mathbb{S}^3$, all belonging to $L_I\cap \mathbb{S}^3$ (see Proposition \ref{permuta}).

\begin{definiz}
A quaternionic torus whose group of biregular automorphisms is isomorphic to $2C_3$ is called a \emph{cyclic torus}.
\end{definiz}

 \begin{prop} \label{2D_4} Let $T$ be a quaternionic torus. 
The group of biregular automorphisms of the torus $T$ is isomorphic to the group $2D_4$, if and only if, there exist $I,J \in \mathbb{S}$ with $J\perp I$,  such that the point $(I, J,JI )$  is  (a representative of) the modulus of $T$.
\end{prop}
\begin{proof} 
Let  $\mathcal B =\{ 1,I, J,JI \}$ be the special  basis associated to $(I, J,JI )\in \mathcal M$, and let  $L$ be the generated lattice such that $T$ is equivalent to $\mathbb{H}/L$. 
Thanks to Proposition \ref{2C_2}, we know that the multiplication by $I$ on the right generates a $2C_2$ subgroup of $Aut(T)$. Using Theorem \ref{generazione}, we are left to prove that the multiplication by $J$ on the right generates a second subgroup of  type $2C_2$ of $Aut(T)$. To this aim notice that the four bases
\begin{eqnarray*}
\mathcal B J^0 &=&\mathcal B =\{ 1,I, J,JI\}=\{ 1,I, J,-K\}\\
\mathcal B J^1 &=&\mathcal B J =\{ J, K, -1, I \}\\
\mathcal B J^2 &=&-\mathcal B  =\{ -1, -I, -J, K \}\\
\mathcal B J^3&=&-\mathcal B J =\{ -J,- K, 1, -I\}\\
\end{eqnarray*}
all generate the lattice $L$. It is then easy to conclude that the subgroup $2D_4\cong \{\pm 1, \pm I, \pm J, \pm K\}$ is a subgroup of the group of automorphisms $Aut(T)$. On the other hand, if $2D_4\subseteq Aut(T)$, then $2C_2\subset 2D_4$ is a subgroup of $Aut(T)$, and hence thanks to Theorem \ref{2C_2}, in the special  basis  $\mathcal B =\{ 1, \alpha_2, \alpha_3, \alpha_4 \}$ that generates $L$, we can choose $I=\alpha_2$ and $\alpha_4=\alpha_3 I$. Theorem \ref{generazione}, Lemma \ref{norma a bis} and Proposition \ref{permuta} imply that the vectors $\{\pm 1,\pm J \}$ must belong to, and generate, the rank-$2$ sublattice $L\cap L_J$ of $L$.  Since $J$ and $-J$ have the same norm of $I$, 
then by the Minkowski-Siegel Reduction Algorithm used to  construct a special basis for $L$, we can suppose that $\alpha_3 =J$ and find the desired basis.  To prove that $2D_4\cong Aut(T)$, begin by noticing that $\mathcal{B}$ is  an orthonormal basis of $\mathbb{H}$; it is then easy to see that $L\cap \mathbb{S}^3=\{ \pm 1, \pm I, \pm J,\pm JI \}$ and that this set has exactly the same cardinality of the group $2D_4$. Proposition \ref{permuta} leads now to the conclusion.
\end{proof}

\begin{definiz}
A quaternionic torus whose group of biregular automorphisms is isomorphic to $2D_4$ is called a \emph{$8$-dihedral torus} (or a \emph{dihedral torus of order $8$}).
\end{definiz}
\begin{oss} As we already mentioned, the notations concerning finite subgroups of unit quaternions vary very much.
We observe that the group $2D_4,$ that we (following \cite[Subsection 3.5]{conway})
called dihedral group of order 8, coincides with the so called multiplicative group of unit quaternions (and not with $D_8$, sometimes called dihedral group of order 8, $D_8 =\langle \langle a, b \rangle\rangle$ with the relations $a^4=b^2=1, bab^{-1}=a^{-1},$ see \cite[Theorem 3.4, page 164]{Artin}).

\end{oss}

\noindent Notice that the lattice of a dihedral torus of order $8$ is generated by the group $2D_4$ and coincides with the ring of Lipschitz quaternions, defined after Proposition \ref{2C_2}.

\begin{prop} \label{2D_6} Let $T$ be a quaternionic torus. 
The group of biregular automorphisms of the torus $T$ is isomorphic to the group $2D_6$, if and only if, there exist $I,J \in \mathbb{S}$, with $J\perp I$,  such that the point $(e^{\frac{\pi I}{3}}, J,J e^{\frac{\pi I}{3}})$ is  (a representative of) the modulus of $T$. 
\end{prop}
\begin{proof} 
Let  $\mathcal B =\{ 1,e^{\frac{\pi I}{3}}, J,J e^{\frac{\pi I}{3}}\}$ be the special  basis associated to  $(e^{\frac{\pi I}{3}}, J,J e^{\frac{\pi I}{3}})\in \mathcal M$ and let $L$ be the generated lattice such that $T$ is equivalent to $\mathbb{H}/L$. Thanks to Proposition \ref{2C_3}, we know that the multiplication by $e^{\frac{\pi I}{3}}$ on the right generates a $2C_3$ subgroup of $Aut(T)$. We are then left to prove that the multiplication by $J$ on the right generates a subgroup of  type $2C_2$ of $Aut(T)$. To this aim notice that the four bases
\begin{eqnarray*}
\mathcal B J^0&=&\mathcal B =\{ 1,e^{\frac{\pi I}{3}}, J,J e^{\frac{\pi I}{3}}\}\\
\mathcal B J^1 &=&\mathcal B J =\{ J,Je^{\frac{-\pi I}{3}}, -1,- e^{\frac{-\pi I}{3}}\}\\
\mathcal B J^2 &=&-\mathcal B  =\{ -1, -e^{\frac{\pi I}{3}}, -J, -Je^{\frac{\pi I}{3}} \}\\
\mathcal B J^3&=&-\mathcal B J =\{ -J,- Je^{\frac{-\pi I}{3}}, 1, e^{\frac{-\pi I}{3}}\}\\
\end{eqnarray*}
all generate the lattice $L$. We then conclude that the dihedral subgroup of order $12$, $2D_6\cong  \{ \pm 1, \pm  e^{\frac{\pi I}{3}}, \pm e^{\frac{2\pi I}{3}}, \pm J, \pm  Je^{\frac{\pi I}{3}}, \pm J e^{\frac{2\pi I}{3}}\}$, is a subgroup of the group of automorphisms $Aut(T)$. On the other hand, if $2D_6\subseteq Aut(T)$, then $2C_3\subset 2D_6$ is a subgroup of $Aut(T)$, and hence, thanks to Proposition \ref{2C_3}, in the special basis  $\mathcal B =\{ 1, \alpha_2, \alpha_3, \alpha_4 \}$ that generates $L$, we can choose $\alpha_2=e^{\frac{\pi I}{3}}$ and $\alpha_4=\alpha_3 e^{\frac{\pi I}{3}}$. Theorem \ref{generazione}, Lemma \ref{norma a bis} and Proposition \ref{permuta}  imply that the vectors $\{\pm 1, \pm J\}$ must belong to, and generate, the rank-$2$ sublattice $L\cap L_J$ of $L$.  Since $J$ and $-J$ have the same norm of $e^{\frac{\pi I}{3}}$, 
then by the Minkowski-Siegel Reduction Algorithm used to construct a special basis for $L$,  we can suppose that $\alpha_3 =J$ and find the desired basis. To prove that $2D_6\cong Aut(T)$, begin by noticing that $L_I$ is orthogonal to $JL_I$; it is then easy to see that $L\cap \mathbb{S}^3=\{ \pm 1, \pm  e^{\frac{\pi I}{3}}, \pm e^{\frac{2\pi I}{3}}, \pm J, \pm  Je^{\frac{\pi I}{3}}, \pm J e^{\frac{2\pi I}{3}}\}$ and that this set has exactly the same cardinality of the group $2D_6$. Proposition \ref{permuta} leads now to the conclusion.   
\end{proof}

\begin{definiz}
A quaternionic torus whose group of biregular automorphisms is isomorphic to $2D_6$ is called a \emph{$12$-dihedral torus} (or a \emph{dihedral torus of order $12$}).
\end{definiz}

\begin{prop} \label{2T} Let $T$ be a quaternionic torus. 
The group of biregular automorphisms of the torus $T$ is isomorphic to the group $2\mathbb{T}$, if and only if, for some $I,J \in \mathbb{S}$ with $J\perp I$, setting $K=IJ$ and $\omega=\frac{-1+I+J+K}{2}$, the point $(-\omega,-I, I\omega)$  is  (a representative of) the modulus of $T$.
\end{prop}
\begin{proof} 
Notice, first of all, that if  $M=\frac{\sqrt 3}{3}(I+J+K)\in \mathbb{S}$, then $\omega=e^{\frac{2\pi M}{3}}$. Consider the special basis  $\mathcal B =\{ 1,-\omega,-I, I\omega \}$ associated to $(-\omega,-I, I\omega)\in \mathcal M$ and let  $L$ be the generated lattice such that $T$ is equivalent to $\mathbb{H}/L$. Thanks to Theorem  \ref{generazione}, we know that  $-\omega$ and $-I$ generate a $2\mathbb{T}$ subgroup of $Aut(T)$. Set  $\tilde \omega=\frac{-1+I-J-K}{2}$, and notice that $\tilde \omega= - \omega-1+I$ belongs to $L$.  It is now necessary (and it is only a direct computation) to verify that the iterated multiplication by powers of $e^{\frac{2\pi M}{3}}$ and by powers of $I$ (on the right) maps the basis $\mathcal B $ onto bases that
generate the lattice $L$.  For example, as for the multiplication by $I$ on the right,  we get  that
\begin{eqnarray*}
\mathcal B (-I)^0 &=&\mathcal B =\{ 1,-\omega, -I, I \omega \}\\
\mathcal B (-I)^1 &=&-\mathcal B I=\{ -I, I\tilde\omega , -1,\tilde\omega\}\\
\mathcal B (-I)^2 &=&-\mathcal B  =\{ -1, \omega, I, -I\omega \}\\
\mathcal B (-I)^3&=&\mathcal B I =\{ I, -I\tilde\omega, 1, -\tilde\omega\}\\
\end{eqnarray*}
are all generating bases for $L$.
We then conclude that the subgroup $2\mathbb{T}=\langle\langle -\omega, -I\rangle\rangle$ is a subgroup of the group of automorphisms $Aut(T)$. On the other hand, if $2\mathbb{T}\subseteq Aut(T)$, then $2C_3\subset 2\mathbb{T}$ is a subgroup of $Aut(T)$, and hence, thanks to Proposition \ref{2C_3}, in the special basis  $\mathcal B =\{ 1, \alpha_2, \alpha_3, \alpha_4 \}$ that generates $L$, we can choose $\alpha_2=-\omega$ and $\alpha_4=\alpha_3 (-\omega)$. Theorem \ref{generazione}, Lemma  \ref{norma a bis} and Proposition \ref{permuta} imply that the vectors $\{\pm1,\pm  I \}$ must belong to, and generate, the rank-$2$ sublattice $L\cap L_I$ of $L$.  Since $I$ and $-I$ have the same norm of $-\omega$, 
then by the Minkowski-Siegel Reduction Algorithm used to construct a special basis for $L$, we can suppose that $\alpha_3 =-I$ and find the desired basis.  To prove that $2\mathbb{T}\cong Aut(T)$, begin by noticing that the $16$ elements $\{\frac{\pm 1\pm I\pm J \pm K}{2}\}$ belong to $L\cap\,\mathbb{S}^3$ as well as the $8$ elements $\{\pm 1, \pm  I, \pm J, \pm K\}$. Therefore the set $L\cap \mathbb{S}^3$ has to have the same cardinality of the group $2\mathbb{T}$. As in the previous cases, at this point Proposition \ref{permuta} leads to the conclusion.   
\end{proof}

\begin{definiz}
A quaternionic torus whose group of biregular automorphisms is isomorphic to $2\mathbb{T}$ is called a \emph{tetrahedral torus}.
\end{definiz}
\noindent Notice that the lattice of the tetrahedral torus is the ring  $\mathcal{H}$ of the Hurwitz quaternions, defined after Proposition \ref{2C_2}. This lattice is generated by the group $2\mathbb{T}$.
\vskip .5cm
The next result can be obtained directly as a consequence of the investigation performed up to now, on the possible groups of biregular automorphisms of ``boundary" tori. Recall that, for a quaternionic torus $T$, in Proposition \ref{permuta} we have introduced the group $A_T=\{ a \in \mathbb{S}^3 \,: \exists \ F\in Aut(T) \  \textnormal{defined as}\   F(q)=qa \}$.

\begin{prop} \label{notabene2}
Let $L$ be a rank-$4$ lattice of $\mathbb{H}$ containing $1.$ Let $T=\mathbb{H} /L$ be the associated quaternionic torus.
When the group $A_T\cong Aut(T)$ 
is not reduced to $\{\pm 1\}$, then it coincides with the biggest subgroup of $L\cap \mathbb{S}^3$.
\end{prop}
\begin{proof}
Suppose $Aut(T)\neq \{\pm 1\}= 2C_1$. Thanks to the classification \eqref{subgroupofS3}, we get that $Aut(T)$ has to coincide with $2C_2, 2C_3, 2D_4, 2D_6$ or $ 2\mathbb{T}$. If $Aut(T)\cong 2C_2$ and $2C_2$ is, by contradiction, strictly contained in a larger subgroup of $\mathbb{S}^3$, then using  Remark \ref{inclusions} we obtain $Aut(T)\cong 2D_4$ or $Aut(T)\cong 2\mathbb{T}$. In both cases $2D_4$ and $2\mathbb{T}$ generate a lattice associated to a different torus (see Propositions \ref{2D_4} and \ref{2T}). If $Aut(T)\cong 2D_4$ and $2D_4$ is, by contradiction, strictly contained in a larger subgroup of $\mathbb{S}^3$, then using Remark \ref{inclusions} we obtain $Aut(T)\cong 2\mathbb{T}$. In this last case $2\mathbb{T}$ generates a lattice associated to a different torus (see Proposition \ref{2T}). In the remaining case in which $Aut(T)\cong 2C_3$ the proof is totally analogous.

\end{proof}
\begin{example}
We give an example which shows, in connection to Proposition \ref{notabene2}, that when $Aut(T)\cong 2C_1= \{\pm 1\}$, then it can be strictly contained in a larger subgroup of $\mathbb{S}^3$.
Let $\{ 1, I, J, K \}$ be the standard basis for the skew field of quaternions.
Consider the lattice $L$ generated by the special basis $\{ 1, I, 3J+\frac{1}{10}, 4K+\frac{1}{100} \}.$ If $T=\mathbb{H}/L,$ then the unitary vectors of $L$ are $\{ \pm 1 , \pm I \}$ and they form a group isomorphic to $2C_2$. 
\end{example}

We conclude this section by stating a summarizing result.

\begin{teo}
The cyclic, cyclic dihedral, $8$-dihedral, $12$-dihedral, tetrahedral tori defined in this section are the unique (up to biregular diffeomorphisms) tori with $Aut(T)\neq \{\pm1\}$.
\end{teo}
\begin{proof}
Follows directly from Propositions \ref{2C_2}, \ref{2C_3}, \ref{2D_4}, \ref{2D_6}, \ref{2T}.
\end{proof}

To close the Section, we recall that the lattices generating tori $T$ with $Aut(T)\cong 2C_2$ or $Aut(T)\cong 2C_3$ are classically called \emph{regular tessellations of $\mathbb{R}^2$}; the lattices generating tori  $T$ with $Aut(T)\cong 2D_4$,  $Aut(T)\cong 2D_6$ or $Aut(T)\cong 2\mathbb{T}$ are called \emph{regular tessellations of $\mathbb{R}^4.$}

\section{Appendix A: an algorithm to check if a basis is reduced} \label{AlgRid}

Let $R=(r_{i,j})_{i,j=1, \cdots, 4}$ be the Gram matrix associated to a given basis $\mathcal{B}=\{v_1,v_2,v_3,v_4 \}$ of the rank-$4$ lattice $L.$ Reordering the four vectors $\{v_1,v_2,v_3,v_4 \}$, without loss of generality, we can always suppose that $r_{1,1}\leq r_{2,2}\leq r_{3,3}\leq r_{4,4}.$ To check that $\mathcal{B}$ is a reduced basis, we will check the  fact that  $R$ is a reduced Gram matrix.

The first step of our algorithm is the easiest:
\begin{itemize}
\item[Step 0:]  We check if  $r_{k,k+1}=\langle v_k, v_{k+1} \rangle$ (for $k=1,2,3$) are all non negative quantities. If this is true, we proceed in the algorithm; otherwise we conclude that the Gram matrix $R$, and the basis $\mathcal{B}$,  are not reduced, and stop.
Since $R$ is a $(4 \times 4)$ symmetric, real and positive definite matrix, there exists a positive definite diagonal matrix 
$$
D=\begin{bmatrix}
\lambda_1^2& 0&0&0\\
0&\lambda_2^2&0&0\\
0&0&\lambda_3^2&0\\
0&0&0&\lambda_4^2
\end{bmatrix}
$$ 
and an orthogonal matrix $Q$ such that $^t Q  R  Q =D.$ Moreover, we can suppose that $0< \lambda_1\leq \lambda_2\leq  \lambda_3\leq  \lambda_4.$
In order to verify if $R$ is a reduced Gram matrix we proceed as follows:
\item[Step 1:] Since $R=Q  D \  ^t Q,$ the quadratic form 
$(n_1,n_2,n_3,n_4) R  \ ^t (n_1,n_2,n_3,n_4)$ can also be written as
$$(n_1,n_2,n_3,n_4) Q  D \  ^t Q  \ ^t (n_1,n_2,n_3,n_4)=\lambda_1^2 x_1^2 + \cdots + \lambda_4^2 x_4^2$$
where $\lambda_1^2, \cdots , \lambda_4^2$ are the ordered positive eigenvalues of $D$ and 
$(x_1 \cdots, x_4)=(n_1, \cdots, n_4)  Q.$ 
The geometric locus of vectors $(x_1, \cdots ,x_4)\in \mathbb{R}^4$ for which the diagonalized quadratic form is equal to $r_{1,1}=\langle v_1, v_1 \rangle$ is an ellipsoid having the length  of the maximal  axis of symmetry equal to $\frac{2|v_1|}{\lambda_1}.$
Therefore, the quadruplets $(n_1,n_2,n_3,n_4)\in \mathbb{Z}^4$ such that 
\begin{equation}\label{Rcondizione 1}
(n_1,n_2,n_3,n_4) R  \ ^t(n_1,n_2,n_3,n_4)< r_{1,1}
\end{equation}
belong necessarily to the finite set $E_1=\mathbb{Z}^4 \cap I_1^4 $ where 
$$I_1=\left[ -\frac{|v_1|}{\lambda_1}, \frac{|v_1|}{\lambda_1}\right].$$
At this point we recall the first step of the construction of the Minkowski-Siegel Reduction Algorithm: we check if there exists a point $(n_1,n_2,n_3,n_4)$ in the finite set $E_1\setminus \{\underline 0\}$ such that inequality \eqref{Rcondizione 1} is fulfilled. If the answer is yes, then we conclude that the Gram matrix $R$, and hence the basis $\mathcal{B}$, are not reduced, and stop. Otherwise we proceed in the algorithm.

\item[Step 2:] In this step we consider the quadratic equation 
$$
(n_1,n_2,n_3,n_4) R  \ ^t(n_1,n_2,n_3,n_4)
=r_{2,2}.
$$
Let 
$$I_2=\left[ -\frac{|v_2|}{\lambda_1}, \frac{|v_2|}{\lambda_1}\right]$$ 
and set $E_2$ to be the finite set $\mathbb{Z}^4 \cap I_2^4$. By Definition \ref{reducedGram} and by condition B2)$'$, a reduced Gram matrix is such that: if there exists $(n_1,n_2,n_3,n_4)\in E_2\setminus \{\underline 0\}$ with
\begin{equation}\label{condizione 2}
(n_1,n_2,n_3,n_4) R  \ ^t(n_1,n_2,n_3,n_4) < r_{2,2}
\end{equation}
then $n_2,n_3,n_4$ have common divisors. Therefore we check if B2)$'$ holds true. If the answer is no, then we conclude that the Gram matrix $R$, and hence the basis $\mathcal{B}$, are not reduced, and stop. Otherwise we proceed in the algorithm. 
\item[Step 3:] Similar  procedure applies to the quadratic equation
$$(n_1,n_2,n_3,n_4) R  \ ^t(n_1,n_2,n_3,n_4)=r_{3,3}.$$
Let 
$$I_3=\left[ -\frac{|v_3|}{\lambda_1}, \frac{|v_3|}{\lambda_1}\right]$$ 
and set $E_3$ to be the finite set $\mathbb{Z}^4 \cap I_3^4$. By Definition \ref{reducedGram} and by condition B2)$'$, a reduced Gram matrix is such that: if there exists $(n_1,n_2,n_3,n_4)\in E_3\setminus \{\underline 0\}$ with
\begin{equation}\label{condizione 3}
(n_1,n_2,n_3,n_4) R  \ ^t(n_1,n_2,n_3,n_4) < r_{3,3}
\end{equation}
then $n_3,n_4$ have common divisors. Therefore we check if B2)$'$ holds true.  If the answer is no, then we conclude that the Gram matrix $R$, 
and hence the basis $\mathcal{B}$, are not reduced, and stop. Otherwise we proceed in the algorithm. 
\item[Step 4:] In the last step our  procedure is applied to the quadratic equation
$$(n_1,n_2,n_3,n_4) R  \ ^t(n_1,n_2,n_3,n_4)=r_{4,4}.$$
Let 
$$I_4=\left[ -\frac{|v_4|}{\lambda_1}, \frac{|v_4|}{\lambda_1}\right]$$ 
and set $E_4$ to be the finite set $\mathbb{Z}^4 \cap I_4^4$. By Definition \ref{reducedGram} and by condition B2)$'$, a reduced Gram matrix is such that: if there exists $(n_1,n_2,n_3,n_4)\in E_4\setminus \{\underline 0\}$ with
\begin{equation}\label{condizione 4}
(n_1,n_2,n_3,n_4) R  \ ^t(n_1,n_2,n_3,n_4) < r_{4,4}
\end{equation}
then $n_4$ must be different from $\pm 1$. Therefore we check if B2)$'$ holds true.  If the answer is no, then we conclude that the Gram matrix $R$, and hence the basis $\mathcal{B}$, are not reduced. Otherwise we finally conclude that the Gram matrix $R$, and hence the basis $\mathcal{B}$, are reduced.
\end{itemize}

\vskip .5cm 

A few significant examples, useful  to illustrate the meaning of the results obtained, are the following. \\
Let $\{ 1, I, J, K \}$ be the standard basis for the skew field of quaternions. 

 \begin{example}
An example of a cyclic-dihedral torus is the one associated to  $(I, e^{\frac{2}{5}\pi J}, e^{\frac{2}{5}\pi J}I)\in \mathcal M$.  Indeed, running the Algorithm presented in this section, we find out that Step 0 is satisfied and that the only vector of the lattice $L$  generated by the special basis $\mathcal B=\{1, I, e^{\frac{2}{5}\pi J}, e^{\frac{2}{5}\pi J}I \}$ inside the unit ball is the null vector; besides, on $\mathbb{S}^3\cap L$ we only find the set of vectors $\mathcal B\cup -\mathcal B$.
\end{example}

\begin{example}
A second example of a cyclic-dihedral torus is the one associated to $(I, 4J+3K, -4K+3J)\in \mathcal M$.  Indeed, running the Algorithm presented in this section, we find out that Step 0 is satisfied and all conditions in B2) are verified. The only vector of the lattice  $L$ generated by  the special basis $\mathcal B=\{1, I, 4J+3K, -4K+3J \}$ inside the unit ball is the null vector; besides, on $\mathbb{S}^3\cap L$ we only find the set of vectors $\{\pm 1, \pm I\}.$  On the sphere of radius $5$ there are 16 elements of $L$ and the product by $a\in \{\pm 1, \pm I\}$ permutes them, as explained in Proposition \ref{permuta}. 
\end{example}


\begin{example}
To present an example of a cyclic torus we use the one associated to $(e^{\frac{\pi}{3}I}, e^{\frac{2}{5}\pi J}, e^{\frac{2}{5}\pi J} e^{\frac{\pi}{3}I})\in \mathcal M$, and hence to the lattice $L$ generated by the special basis $\mathcal B=\{1,e^{\frac{\pi}{3}I}, e^{\frac{2}{5}\pi J}, e^{\frac{2}{5}\pi J} e^{\frac{\pi}{3}I} \}$.  Indeed, running the Algorithm presented in this section, we find out that Step 0 is satisfied and the only vector of the lattice $L$  inside the unit ball is the null vector; besides, on $L \cap \mathbb{S}^3$ we have $12$ vectors: $8$ of them are in $\mathcal B\cup -\mathcal B$ and 4 other vectors correspond to $\{\pm e^{\frac{2}{3}\pi I}, \pm e^{\frac{2}{5}\pi J}e^{\frac{2}{3}\pi I}\}$.
\end{example}

\section{Appendix B: an algorithm to check if a basis is tame}

We want now to provide an algorithm to establish when a reduced basis is a tame basis.

Let $R=(r_{i,j})_{i,j=1, \cdots, 4}$ be the reduced Gram matrix associated to a reduced basis $\mathcal{B}=\{v_1,v_2,v_3,v_4 \}$ of a rank-$4$ lattice $L.$ We will use precisely the same notations as in the algorithm of the previous section. The first step of our new algorithm is the following:
\begin{itemize}
\item[Step 0:]  We check if  $r_{k,k+1}=\langle v_k, v_{k+1} \rangle$ (for $k=1,2,3$) are all strictly positive quantities. If this is true, we proceed in the algorithm; otherwise we conclude that the basis $\mathcal{B}$ is not tame, and stop.
\item[Step 1:] Since $R=Q  D  ^t Q,$ the quadratic form 
$(n_1,n_2,n_3,n_4) R  \ ^t (n_1,n_2,n_3,n_4)$ can also be written as
$$(n_1,n_2,n_3,n_4) Q  D  \ ^t Q  \ ^t (n_1,n_2,n_3,n_4)=\lambda_1^2 x_1^2 + \cdots + \lambda_4^2 x_4^2$$
where $\lambda_1^2, \cdots , \lambda_4^2$ are the ordered positive eigenvalues of $D$ and 
$(x_1 \cdots, x_4)=(n_1, \cdots, n_4)  Q.$ 
The geometric locus of vectors $(x_1, \cdots ,x_4)\in \mathbb{R}^4$ for which the diagonalized quadratic form is equal to $r_{1,1}=\langle v_1, v_1 \rangle$ is an ellipsoid having the length  of the maximal axis of symmetry equal to $\frac{2|v_1|}{\lambda_1}.$
Therefore, the quadruplets $(n_1,n_2,n_3,n_4)\in \mathbb{Z}^4$ such that 
\begin{equation}\label{condizione 1}
(n_1,n_2,n_3,n_4) R  \ ^t(n_1,n_2,n_3,n_4) = r_{1,1}
\end{equation}
belong necessarily to the finite set $E_1=\mathbb{Z}^4\cap  I_1^4$, where 
$$I_1=\left[ -\frac{|v_1|}{\lambda_1}, \frac{|v_1|}{\lambda_1}\right].$$
At this point we recall the first step of the construction of the Minkowski-Siegel Reduction Algorithm: we check if there exists a point $(n_1,n_2,n_3,n_4)$ in the finite set $E_1\setminus \{(\pm 1,0,0,0)\}$ such that equality \eqref{condizione 1} is fulfilled. If the answer is yes, then the integers $(n_1,n_2,n_3,n_4)$ have common divisors: indeed, if no non-trivial common divisor exists, we can find a vector in the lattice $L$ whose squared norm is equal to $r_{1,1}$ and which can substitute $v_1$ in $\mathcal{B}.$ If this is the case the basis $\mathcal{B}$ is not unique and it is not tame. Otherwise we proceed in the algorithm.

\item[Step 2:] In this step we consider the quadratic equation 
$$
(n_1,n_2,n_3,n_4) R  \ ^t(n_1,n_2,n_3,n_4)
=r_{2,2}.
$$
Let 
$$I_2=\left[ -\frac{|v_2|}{\lambda_1}, \frac{|v_2|}{\lambda_1}\right]$$ 
and set $E_2$ to be the finite set $\mathbb{Z}^4 \cap I_2^4$. 
Suppose there exists  no quadruplet $(n_1,n_2,n_3,n_4)\in E_2\setminus \{(0,\pm 1,0,0)\}$ with
\begin{equation}\label{condizione 21}
(n_1,n_2,n_3,n_4) R  \ ^t(n_1,n_2,n_3,n_4) = r_{2,2}.
\end{equation}
Then we go to the next step of the algorithm. Suppose that instead we find a (finite) set $B$ of  quadruplets $(n_1,n_2,n_3,n_4)\in E_2\setminus \{(0,\pm 1,0,0)\}$ with
\begin{equation}\label{condizione 22}
(n_1,n_2,n_3,n_4) R  \ ^t(n_1,n_2,n_3,n_4) = r_{2,2}.
\end{equation}
We use now Definition \ref{tameness}, Proposition \ref{baseunica} and condition B2)$'$:  if,  for all elements $(n_1,n_2,n_3,n_4)\in B$, $(n_2,n_3,n_4)$ have  common divisors,  we go to the next step of the algorithm. Otherwise we conclude that the basis $\mathcal{B}$ is not tame, and stop. 
\item[Step 3:] Similar procedure applies to the quadratic equation
$$(n_1,n_2,n_3,n_4) R  \ ^t(n_1,n_2,n_3,n_4)=r_{3,3}.$$
Let 
$$I_3=\left[ -\frac{|v_3|}{\lambda_1}, \frac{|v_3|}{\lambda_1}\right]$$ 
and set $E_3$ to be the finite set $\mathbb{Z}^4 \cap I_3^4$. 
Suppose there exists  no quadruplet $(n_1,n_2,n_3,n_4)\in E_3\setminus \{(0, 0, \pm 1,0)\}$ with
\begin{equation}\label{condizione 23}
(n_1,n_2,n_3,n_4) R  \ ^t(n_1,n_2,n_3,n_4) = r_{3,3}.
\end{equation}
Then we go to the next step of the algorithm. Suppose that instead we find a (finite) set $C$ of quadruplets $(n_1,n_2,n_3,n_4)\in E_3\setminus \{(0, 0,\pm 1,0)\}$ with
\begin{equation}\label{condizione 24}
(n_1,n_2,n_3,n_4) R  \ ^t(n_1,n_2,n_3,n_4) = r_{3,3}.
\end{equation}
We use again Definition \ref{tameness}, Proposition \ref{baseunica} and condition B2)$'$:  
if,  for all elements $(n_1,n_2,n_3,n_4)\in C$, $(n_3,n_4)$ have  common divisors,  we go to the next step of the algorithm. Otherwise we conclude that the basis $\mathcal{B}$ is not tame, and stop. 

\item[Step 4:] In the last step our procedure is applied to the quadratic equation
$$(n_1,n_2,n_3,n_4) R  \ ^t(n_1,n_2,n_3,n_4)=r_{4,4}.$$
Let 
$$I_4=\left[ -\frac{|v_4|}{\lambda_1}, \frac{|v_4|}{\lambda_1}\right]$$ 
and set $E_4$ to be the finite set $\mathbb{Z}^4 \cap I_4^4$. 
Suppose there exists  no quadruplet $(n_1,n_2,n_3,n_4)\in E_4\setminus \{(0, 0, 0, \pm 1)\}$ with
\begin{equation}\label{condizione 25}
(n_1,n_2,n_3,n_4) R  \ ^t(n_1,n_2,n_3,n_4) = r_{4,4}.
\end{equation}
Then we conclude that $\mathcal{B}$ is tame. Suppose that instead we find a (finite) set $D$ of quadruplets $(n_1,n_2,n_3,n_4)\in E_4\setminus \{(0, 0,0, \pm 1)\}$ with
\begin{equation}\label{condizione 26}
(n_1,n_2,n_3,n_4) R  \ ^t(n_1,n_2,n_3,n_4) = r_{4,4}.
\end{equation}
We use again Definition \ref{tameness}, Proposition \ref{baseunica} and condition B2)$'$: if,  for all elements  $(n_1,n_2,n_3,n_4)\in D$, $n_4=\pm 1$, then the basis $\mathcal{B}$ is not tame. Otherwise we conclude that the basis $\mathcal{B}$ is tame. 
\end{itemize}
\vskip .5cm
\vskip .5cm

To conclude, we present an explicit example of special tame lattice (i.e. a lattice whose reduced Gram matrix belongs to $\mathring{\mathcal R}$). \\
Let $\{ 1, I, J, K \}$ be the standard basis for the skew field of quaternions.

\begin{example} The lattice $L$ generated by the special basis
\begin{equation*} 
\mathcal{B} = \{1, 2I+\frac{1}{10}, 3J+\frac{1}{100}, 4K+\frac{1}{1000} \}
\end{equation*}
is a tame lattice. The proof follows by a direct application of the algorithm. 
\end{example}

\bibliographystyle{amsplain}

\end{document}